\documentclass[10pt]{amsart}

\usepackage{amsmath, amsfonts, amssymb, amsthm, mathrsfs}

\usepackage{tikz-cd}

\usepackage[colorlinks, urlcolor=black]{hyperref}
\hypersetup{citecolor=blue, linkcolor=blue}

\newtheorem{theorem}{Theorem}[section]
\newtheorem{definition}[theorem]{Definition}
\newtheorem{proposition}[theorem]{Proposition}
\newtheorem{conjecture}[theorem]{Conjecture}
\newtheorem{corollary}[theorem]{Corollary}
\newtheorem{lemma}[theorem]{Lemma}
\newtheorem{remark}[theorem]{Remark}

\numberwithin{equation}{section}
\newcommand{\tI}{{}^{\otimes}{\rm I}}
\makeatletter
\newcommand{\loccit}[1]{\@ifundefined{hyperlink}{[\emph{loc.\,cit.}]}{\hyperlink{cite.#1}{[\emph{loc.\,cit.}]}}}
\makeatother

\author{Luis Alberto Lomel\'i}
\address{Luis Alberto Lomel\'i \\ Instituto de Matem\'aticas \\ Pontificia Universidad Cat\'olica de Valpara\'iso \\  Blanco Viel 596 \\ Cerro Bar\'on \\ Valpara\'iso \\ Chile}
\email{luis.lomeli@pucv.cl}


\begin{document}

\title[Ramanujan and Riemann for the unitary groups]{Ramanujan Conjecture and Riemann Hypothesis \\ for the unitary groups over function fields}

\begin{abstract}
We establish generic local and global Langlands functoriality, namely stable base change, for the unitary groups in characteristic $p$. Main applications are the Ramanujan Conjecture and the Riemann Hypothesis for extended Langlands-Shahidi $L$-functions of the unitary groups over function fields.
\end{abstract}

\maketitle

\section*{Introduction}

We study global Langlands functoriality, or more precisely, we establish stable base change for globally generic representations of the quasi-split unitary groups ${\rm U}_N$ over function fields. We rely on a study of $L$-functions via the Langlands-Shahidi (LS) method in positive characteristic. For the local theory, there is a well known proof in algebraic number theory of the existence theorem of local class field theory from the global theory. Analogously, we here prove the existence of generic local base change via a global to local argument.

In the absence of an invariant stabilized trace formula in positive characteristic, our approach combines the LS method \cite{LomLS} with the converse theorem of Cogdell and Piatetski-Shapiro \cite{CoPS1994}. Globally, we build on prior work on the classical groups over function fields \cite{Lo2009}, and we guide ourselves with the work of Kim and Krishnamurthy for the unitary groups in the case of number fields \cite{KiKr2004, KiKr2005}.

We observe that over number fields, functoriality for globally generic representations of the classical groups was established by Cogdell, Kim, Piatetski-Shapiro and Shahidi \cite{CKPSS2004}. The work of Arthur on endoscopic classification establishes the general case for not necessarily generic representations in \cite{Ar2013}, followed by the work of Mok on the unitary groups \cite{Mo2015}.

Over function fields, we prove for the unitary groups an analogue of the Weil conjectures for globally generic automorphic $L$-functions. Via base change, we are able to transport two important problems from unitary groups to general linear groups. Indeed, by combining our results with those of L.~Lafforgue \cite{La2002}, we prove the Ramanujan Conjecture, in addition to the Riemann Hypothesis for extended Langlands-Shahidi automorphic $L$-functions for the unitary groups. More precisely, Theorem~\ref{mainthm} shows that these $L$-functions are rational, satisfy a functional equation, and have their zeros on the line $\Re(s) = 1/2$. Furthermore, Theorem~\ref{ramanujan} states that if $\pi$ is a globally generic cuspidal automorphic representation of a unitary group, then every local component $\pi_v$ is tempered; in particular, whenever $\pi_v$ is unramified, its Satake parameters have absolute value $1$.

Let us now describe our study of $L$-functions and Langlands functoriality in characteristic $p$, beginning in \S~\ref{sec:unitary}. We work with quasi-split odd or even unitary groups of rank $n$, i.e., ${\rm U}_N$, with $N = 2n+1$ or $2n$, and with the restriction of scalars ${\rm Res}\,{\rm GL}_m$ of a general linear group (where we use the underlying quadratic extension in common with ${\rm U}_N$). After some preliminaries, we recall Asai $L$-functions, which play an important role in the LS method for the unitary groups \cite{HeLo2013, Lo2016}, and in the description of the image of base change. Asai and twisted Asai $L$-functions are connected to Rankin-Selberg product $L$-functions, see \S~\ref{RSAsai}.

The defining properties of extended Langlands-Shahidi local factors in our case can be found in \S~\ref{sec:extended}. The list of ten axioms uniquely characterizes the system of $\gamma$-factors, $L$-functions and root numbers. Indeed, to prove uniqueness, we can proceed as in Theorem~4.3 of \cite{Lo2015}; we only observe that we may now use Lemma~3.1 of \cite{LomLS} to avoid the use of stability of $\gamma$-factors in the proof, as we do in Theorem~5.1 of \loccit{LomLS}. Existence for products of generic representations of ${\rm U}_N$ and ${\rm Res}\,{\rm GL}_m$ follows from \loccit{LomLS}, a case that is needed for us to establish base change, before extending LS-factors to include products of two unitary groups. Moreover, our main result on $L$-functions, Theorem~\ref{mainthm}, includes existence and uniqueness of an extended system of factors for generic representations.

In \S~\ref{sec:BC}, we establish a preliminary ``weak'' base change (agreeing with local base change at every unramified place) for globally generic representations, Proposition~\ref{weakBC}. Given a globally generic cuspidal automorphic representation $\pi$ of ${\rm U}_N$, we construct a candidate representation $\Pi$ for the base change to ${\rm Res}\,{\rm GL}_N$ place by place: at unramified places via the Satake correspondence \cite{Mi2011}; at split places by $\Pi_v = \pi_v \otimes \tilde{\pi}_v$; and at the finitely many remaining places by any representation with the same central character as $\pi_v$, thanks to stability under highly ramified twists. We are then able to obtain a base change by applying the converse theorem to $\Pi$, possible, due to the fact that all Langlands-Shahidi automorphic $L$-functions over function fields are \emph{nice} \cite{Lom19}, i.e., they satisfy a functional equation and are rational. In fact, if we incorporate twists by suitable characters, global LS L-functions become polynomials.

In \S\S~\ref{sec:local} and \ref{sec:LLF} we turn towards the local Langlands conjecture for the unitary groups, that is, local base change. Let ${\rm U}_N$ be defined over a non-archimedean local field $k$ of characteristic $p$ with underlying separable quadratic extension $l/k$. Then, Theorem~\ref{localLC} establishes the local transfer
\begin{equation*}
   \left\{ \begin{array}{c} \text{generic representations} \\ \pi \text{ of } {\rm U}_N(k) \end{array} \right\}
   \stackrel{\rm BC}{\longrightarrow}
   \left\{ \begin{array}{c} \text{generic representations} \\ \Pi \text{ of } {\rm GL}_N(l) \end{array} \right\}
\end{equation*}
The local base change $\Pi = {\rm BC}(\pi)$ is known as stable base change. It is uniquely characterized by the property that it preserves local $L$-functions, $\gamma$-factors and root numbers, as in the case of ${\rm GL}_N$ \cite{He1993}. The proof begins with a global to local argument, where we use the weak global base change of Proposition~\ref{weakBC} to deduce existence for generic supercuspidals of ${\rm U}_N(k)$; that these lifts are tempered is then obtained by a purely local argument, Proposition~\ref{scTempered}. We then go through the classification of generic representations for the unitary groups. In particular, the construction of discrete series by M\oe glin and Tadi\'c \cite{MoTa2002} plays an important role; their Basic Assumption (BA) is part of Theorem~5.4 of \cite{LomLS}, where we follow Shahidi for generic representations (in general, we verify (BA) in \cite{GaLo2018} without the generic assumption). Note that \cite{dCHeLo} settles the standard module conjecture in characteristic $p$. In addition, the work of M.~Tadi\'c \cite{Ta1986} on the classification of unitary representations for ${\rm GL}_m$ is very useful. Local base change in general is thus established recursively: the Langlands classification reduces to the tempered case; then, tempered representations are constructed via discrete series, which in turn are constructed via supercuspidals. In this article, we present complete proofs for generic representations, and we refer to \S\S~7 and 8 of \cite{GaLo2018} for a discussion of local base change for general admissible representations.

Let $E/F$ be a separable quadratic extension of global function fields. In \S~\ref{sec:strong}, we establish ``strong'' base change (agreeing with local base change at every place), Theorem~\ref{strongBC}, for globally generic representations of unitary groups:
\begin{equation*}
   \left\{ \begin{array}{c} \text{globally generic cuspidal} \\ \text{automorphic representations} \\ \pi \text{ of } {\rm U}_N(\mathbb{A}_F) \end{array} \right\}
   \stackrel{\rm BC}{\longrightarrow}
   \left\{ \begin{array}{c} \text{automorphic representations} \\ \Pi = \Pi_1 \boxplus \cdots \boxplus \Pi_d \\ \text{of } {\rm GL}_N(\mathbb{A}_E) \end{array} \right\}
\end{equation*}
Using the analytic properties of automorphic $L$-functions over function fields, we show that $\Pi = {\rm BC}(\pi)$ is an isobaric sum of pairwise distinct unitary cuspidal automorphic representations $\Pi_i$ of ${\rm GL}_{n_i}(\mathbb{A}_E)$, each satisfying $\Pi_i \cong \tilde{\Pi}_i^\theta$.

In \S~\ref{sec:RR} we conclude our study of $L$-functions for the unitary groups. For representations that are not necessarily generic, we note in \S~\ref{nongeneric} how to reduce the study of local factors to the case of generic representations, under the assumption that Shahidi's tempered $L$-packet conjecture is valid, Conjecture~\ref{temperedconj}. This conjecture is known in characteristic zero, and we refer to Ganapathy-Varma \cite{GaVa2017} for the validity of Conjecture~\ref{temperedconj} for the split classical groups in characteristic $p$. In \S~\ref{proofmain}, we complete the proof of Theorem~\ref{mainthm}, which includes the Riemann Hypothesis for our automorphic $L$-functions. And, we establish the Ramanujan Conjecture in Theorem~\ref{ramanujan}.

\subsection*{Acknowledgments} The author would like to thank professors F.~Shahidi and G.~Henniart for many mathematical discussions. Mathematical conversations with J.~Bernstein at MSRI and MPIM were particularly insightful. The author would like to thank H.~del Castillo, G.~Harder, Y.~Kim, M.~Krishnamurthy, P.~Kutzko, L.~Lafforgue, V.~Lafforgue, E.~Lapid, A.~M\'inguez, C.~M\oe glin, J.~Navarro, D.~Prasad, A.~Roche, R.~Schmidt, C.~Sch\"omann, S.~Varma and M.-F.~Vign\'eras. He is very grateful to the following institutions: Mathematical Sciences Research Institute, Max-Planck-Institut f\"ur Mathematik and Institut des Hautes \'Etudes Scientifiques. Work on this article was supported in part by MSRI NSF Grant DMS 0932078, Project VRIEA-PUCV 039.367 and FONDECYT Grant 1171583. The author acknowledges the use of Claude (Anthropic) for editorial and typesetting assistance during the final preparation of this manuscript.

\section{The quasi-split unitary groups and Asai \texorpdfstring{$L$}{L}-functions}\label{sec:unitary}

We give some preliminaries on unitary groups and their $L$-groups. We then recall the induction step of the LS method for the unitary groups, where we obtain Asai $L$-functions; we collect some of their basic properties and their connection to Rankin-Selberg $L$-functions.

\subsection{Unitary groups}\label{unitarygroups} Let $E$ be a degree-$2$ finite \'etale algebra over a field $F$ with non-trivial involution $\theta$. For a commutative $F$-algebra $R$, we let $\theta$ act on $E \otimes_F R$ via $\theta \otimes 1$. We write $\bar{x} = \theta(x)$, for $x \in E \otimes_F R$, and extend conjugation to elements $g = (g_{i,j})$ of ${\rm GL}_n(E \otimes_F R)$, i.e., $\bar{g} = (\bar{g}_{i,j})$. Let $N = 2n+1$ or $2n$. We fix the form
\begin{equation*}
   h_N(x,y) = \sum_{i=1}^{N} (-1)^{i-1} \bar{x}_i y_{N+1-i}, \quad x, y \in (E \otimes_F R)^N,
\end{equation*}
which is hermitian if $N$ is odd and skew-hermitian if $N$ is even; in the latter case, $\zeta h_N$ is hermitian for any $\zeta \in E^\times$ of trace zero, and defines the same unitary group.

We then have odd or even quasi-split unitary groups ${\rm U}_N$, defined over $F$ and of rank $n$ when $E$ is a field, whose group of $R$-points is given by
\begin{equation*}
   {\rm U}_N(R) = \left\{ g \in {\rm GL}_N(E \otimes_F R) \, | \, h_N(gx,gy) = h_N(x,y) \right\}.
\end{equation*}
Similarly, the restriction of scalars ${\rm Res}_{E/F}{\rm GL}_n$ has group of $R$-points ${\rm GL}_n(E \otimes_F R)$. Given a representation $\pi$ of ${\rm GL}_n(E \otimes_F R)$, we let $\pi^\theta$ be the representation given by $\pi^\theta(g) = \pi(\bar{g})$; its contragredient representation is $\tilde{\pi}^\theta$. We use this mainly for $R = F$, for a completion $F_v$ of a global field $F$, where $E \otimes_F F_v = E_v$, and for the ad\`eles $\mathbb{A}_F$, where $E \otimes_F \mathbb{A}_F = \mathbb{A}_E$.

In particular, we have two cases to which every degree-$2$ finite \'etale algebra over $F$ is isomorphic: $E$ the separable algebra $F \times F$, where we have $\theta(x) = \theta(x_1,x_2) = (x_2,x_1) = \bar{x}$ and ${\rm N}_{E/F}(x_1,x_2) = x_1x_2$; and, $E/F$ a separable quadratic extension, where we have ${\rm Gal}(E/F) = \left\{ 1, \theta \right\}$ and ${\rm N}_{E/F}(x) = x \bar{x}$. Notice that
\begin{equation*}
   {\rm U}_1(F) = E^1 = \ker({\rm N}_{E/F}),
\end{equation*}
where in the separable algebra case we embed $F \hookrightarrow E$ via $F \cong \left\{ (x,x) \in E | x \in F \right\}$ and $F^\times \hookrightarrow E^\times$ via $F^\times \cong \left\{ (x,x^{-1}) \in E | x \in F^\times \right\} = {\rm U}_1(F)$. In both cases, Hilbert's theorem 90 gives us a continuous surjection
\begin{equation}\label{hilbert90}
   \mathfrak{h} : E^\times \twoheadrightarrow E^1, \ x \mapsto x \bar{x}^{-1}.
\end{equation}

Throughout this article we let ${\bf G}$ be either ${\rm Res}_{E/F}{\rm GL}_n$ or a quasi-split unitary group ${\rm U}_N$ of rank $n$, given by its $R$-points as above, for a degree-$2$ finite \'etale algebra with involution $E$ over $F$. We write $G = {\bf G}(F)$ for its group of rational points.

Notice that in the case of the separable algebra $E = F \times F$, we have
\begin{equation*}
   {\rm U}_N(F) \cong {\rm GL}_N(F) \text{ and } {\rm Res}_{E/F}{\rm GL}_N(F) \cong {\rm GL}_N(F) \times {\rm GL}_N(F).
\end{equation*}
In this case, $\pi^\theta = \pi_2 \otimes \pi_1$ for $\pi = \pi_1 \otimes \pi_2$. This applies, in particular, to the completions $E_v = F_v \times F_v$ at the places $v$ of a global field $F$ which split in $E$.

\begin{remark}\label{alt:def:ueven}
For $N = 2n$, one can also define the quasi-split unitary group via the skew-hermitian form
\begin{equation*}
   h_{2n}'(x,y) = \sum_{i=1}^{n} \bar{x}_i y_{2n+1-i} - \sum_{i=1}^{n} \bar{x}_{2n+1-i} y_i, \quad x, y \in E^{2n}.
\end{equation*}
If we let
\begin{equation*}
   t = {\rm diag}(1,-1,\ldots,(-1)^{n-1},1,\ldots,1) \in {\rm GL}_{2n}(F),
\end{equation*}
then $h_{2n}(x,y) = h_{2n}'(tx,ty)$, and $g \mapsto tgt^{-1}$ gives an isomorphism of ${\rm U}_{2n}$ onto the unitary group of $h_{2n}'$, defined over $F$. It induces a correspondence between standard parabolic subgroups, compatibly with their standard Levi decompositions.
\end{remark}

\subsection{\texorpdfstring{$L$}{L}-groups} Let $E/F$ be a separable quadratic extension of global function fields and let $N = 2n+1$ or $2n$. Then, the $L$-group of ${\rm U}_N$ has connected component ${}^L{\rm U}_N^\circ = {\rm GL}_N(\mathbb{C})$; the $L$-group itself is given by the semidirect product
\begin{equation*}
   {}^L{\rm U}_N = {\rm GL}_N(\mathbb{C}) \rtimes \mathcal{W}_F.
\end{equation*}
A general reference is \cite{Bo1979}; however, we use the Weil form of the $L$-group. To describe the action of the Weil group, let $\Phi_N$ be the $N \times N$ matrix with $ij$-entries $\left( (-1)^{i-1}\delta_{i,N-j+1} \right)$, that is, the matrix of the form $h_N$. Then, the Weil group $\mathcal{W}_F$ acts on ${}^L{\rm U}_N$ through the quotient $\mathcal{W}_F/\mathcal{W}_E \cong {\rm Gal}(E/F) = \left\{ 1, \theta \right\}$ via the outer automorphism
\begin{equation*}
   \theta(g) = \Phi_N^{-1} \, {}^t g^{-1} \Phi_N.
\end{equation*}

\begin{remark}
The form $h_{2n}'$ of Remark~\ref{alt:def:ueven} has associated matrix
\[ J_{2n}' = \left( \begin{array}{cc} & J_n \\ -J_n & \end{array} \right), \]
with $J_n$ the $n \times n$ matrix with $ij$-entries $(\delta_{i,n-j+1})$. Then, the Weil group acts on ${\rm GL}_{2n}(\mathbb{C})$ via $\theta'(g) = J_{2n}'^{-1} \, {}^t g^{-1} J_{2n}'$. Since $\Phi_{2n} = t J_{2n}' t$, we have
\begin{equation*}
   \theta'(tgt^{-1}) = t \, \theta(g) \, t^{-1},
\end{equation*}
so that $(g,w) \mapsto (tgt^{-1},w)$ is an isomorphism between the two $L$-groups. The two actions are thus conjugate under the inner automorphism ${\rm Ad}(t)$, the dual counterpart of the isomorphism of Remark~\ref{alt:def:ueven}, i.e., $\theta' = {\rm Ad}(t) \circ \theta \circ {\rm Ad}(t)^{-1}$. For $n \geq 2$, only $\theta$ preserves the standard pinning of ${\rm GL}_{2n}(\mathbb{C})$. Since $t$ is diagonal, this identification preserves semisimple conjugacy classes, Satake parameters and the base change map, and hence the $L$-functions considered in this article.
\end{remark}

The Langlands base change lift that we will obtain is from a unitary group to the restriction of scalars group ${\bf H}_N = {\rm Res}_{E/F}{\rm GL}_N$. The corresponding $L$-group is given by
\begin{equation*}
   {}^LH_N = {\rm GL}_N(\mathbb{C}) \times {\rm GL}_N(\mathbb{C}) \rtimes \mathcal{W}_F,
\end{equation*}
where the Weil group $\mathcal{W}_F$ acts on ${\rm GL}_N(\mathbb{C}) \times {\rm GL}_N(\mathbb{C})$ through the quotient $\mathcal{W}_F/\mathcal{W}_E \cong {\rm Gal}(E/F) = \left\{ 1, \theta \right\}$ via
\begin{equation*}
   \theta(g_1 \times g_2) = g_2 \times g_1.
\end{equation*}

\subsection{Asai \texorpdfstring{$L$}{L}-functions} Assume first that $l/k$ is a quadratic extension of non-archimedean local fields. Let $\pi$ be an irreducible smooth representation of ${\rm GL}_n(l)$ and let $r_\mathcal{A}$ be the Asai representation
\begin{equation*}
   r_\mathcal{A} : {}^L{\rm Res}_{l/k}{\rm GL}_n \rightarrow {\rm GL}_{n^2}(\mathbb{C}),
\end{equation*}
given by
\begin{equation*}
   r_\mathcal{A}(x,y,1) = x \otimes y, \text{ and } r_\mathcal{A}(x,y,\theta) = y \otimes x.
\end{equation*}
Asai $\gamma$-factors arise in connection with the Langlands-Shahidi method, when we study the Siegel Levi ${\bf M} \cong {\rm Res}_{l/k}{\rm GL}_n$ of a unitary group ${\rm U}_N$, with $N = 2n$. Existence of the $\gamma$-factor
\begin{equation*}
   \gamma(s,\pi,r_\mathcal{A},\psi)
\end{equation*}
in positive characteristic is established in \cite{Lo2016}, where we also have Asai $L$-functions $L(s,\pi,r_\mathcal{A})$ and root numbers $\varepsilon(s,\pi,r_\mathcal{A},\psi)$.

Second, we have the case of $l$ the degree-$2$ finite \'etale algebra $k \times k$, where $\pi = \pi_1 \otimes \pi_2$ and Proposition~4.5 of \loccit{Lo2016} gives that
\begin{equation*}
   \gamma(s,\pi,r_\mathcal{A},\psi) = \gamma(s,\pi_1 \times \pi_2,\psi),
\end{equation*}
a Rankin-Selberg $\gamma$-factor. And, similarly for the corresponding $L$-functions and root numbers.

Asai local factors obtained via the Langlands-Shahidi method are indeed the correct ones. Theorem~3.1 of \cite{HeLo2013}, which we now state, establishes their compatibility with the local Langlands correspondence of \cite{LRS1993}.

\begin{theorem}[Henniart-Lomel\'i]\label{HLAsai}
Let $l/k$ be a separable quadratic extension of non-archimedean local fields of characteristic $p$, and let $\pi$ be an irreducible smooth representation of ${\rm GL}_n(l)$. Let $\sigma$ be the Weil-Deligne representation of $\mathcal{W}_l$ corresponding to $\pi$ via the local Langlands correspondence. Then
\begin{equation*}
   \gamma(s,\pi,r_\mathcal{A},\psi) = \gamma_k^{\rm Gal}(s,\tI(\sigma),\psi).
\end{equation*}
Here, $\tI(\sigma)$ denotes the representation of $\mathcal{W}_k$ obtained from $\sigma$ via tensor induction and the Galois $\gamma$-factors on the right hand side are those of Deligne and Langlands. Local $L$-functions and root numbers satisfy
\begin{align*}
   L(s,\pi,r_\mathcal{A}) &= L(s,\tI(\sigma)), \\
   \varepsilon(s,\pi,r_\mathcal{A},\psi) &= \varepsilon(s,\tI(\sigma),\psi).
\end{align*}
\end{theorem}

\begin{remark}
Since we are in the case of ${\rm GL}_n$, the results of this section hold when $\pi$ is an irreducible smooth representation, and not just generic \cite{HeLo2013}. Furthermore, the Rankin-Selberg products of ${\rm GL}_m$ and ${\rm GL}_n$ that appear in this article arise in the context of the Langlands-Shahidi method in positive characteristic \cite{HeLo2013a}. These are equivalent to those obtained via the integral representation of \cite{JPSS1983}.
\end{remark}

The case of the Siegel Levi ${\bf M}$ of an odd unitary group ${\rm U}_{2n+1}$, i.e., $M \cong {\rm GL}_n(l) \times l^1$, with $l/k$ a quadratic extension of non-archimedean local fields, leads to twisted Asai $\gamma$-factors. Let $\eta_{l/k}$ be the character of class field theory corresponding to $l/k$, and let $\tau$ be an irreducible smooth representation of ${\rm GL}_n(l)$. Then, the existence in positive characteristic of
\begin{equation*}
   \gamma(s,\tau \otimes \eta_{l/k},r_\mathcal{A},\psi)
\end{equation*}
is also established in \cite{Lo2016}, in addition to the $L$-functions $L(s,\tau \otimes \eta_{l/k},r_\mathcal{A})$ and root numbers $\varepsilon(s,\tau \otimes \eta_{l/k},r_\mathcal{A},\psi)$.

Furthermore, the case $l = k \times k$ is also discussed in \loccit{Lo2016}, where $\tau = \tau_1 \otimes \tau_2$ and each $\tau_i$ is an irreducible smooth representation of ${\rm GL}_n(k)$. We thus have
\begin{equation*}
   \gamma(s,\tau \otimes \eta_{l/k},r_\mathcal{A},\psi) = \gamma(s,\tau_1 \times \tau_2,\psi),
\end{equation*}
where on the right hand side we have a Rankin-Selberg $\gamma$-factor.

\subsection{Rankin-Selberg products and Asai factors}\label{RSAsai} Given an irreducible smooth representation $\pi$ of ${\rm GL}_n(l)$, let $\pi^\theta$ be as in \S~\ref{unitarygroups}. Let $\sigma$ be the $n$-dimensional $\ell$-adic Frob-semisimple Weil-Deligne representation of $\mathcal{W}_l$ corresponding to $\pi$ via the local Langlands correspondence. Write $\sigma^\theta$ to denote the representation $w \mapsto \sigma(w_\theta w w_\theta^{-1})$, for $w_\theta \in \mathcal{W}_k \setminus \mathcal{W}_l$; notice that its isomorphism class does not depend on the choice of $w_\theta$.

\begin{proposition}
Let $\pi$ be an irreducible smooth representation of ${\rm GL}_n(l)$. Then
\begin{align}
   \gamma(s,\pi,r_\mathcal{A},\psi) &= \gamma(s,\pi^\theta,r_\mathcal{A},\psi),  \\
   \gamma(s,\pi \otimes \eta_{l/k},r_\mathcal{A},\psi) &= \gamma(s,\pi^\theta \otimes \eta_{l/k},r_\mathcal{A},\psi), 
\end{align}
and we have the following equation involving Rankin-Selberg and Asai $\gamma$-factors
\begin{equation}\label{asaiRS}
   \gamma(s,\pi \times \pi^\theta,\psi_l) = \gamma(s,\pi,r_\mathcal{A},\psi) \gamma(s,\pi \otimes \eta_{l/k},r_\mathcal{A},\psi).
\end{equation}
\end{proposition}
\begin{proof}
With $\sigma$ as above, the local Langlands correspondence matches $\sigma^\theta$ with $\pi^\theta$. Moreover, from the definition of tensor induction (see \cite{CuRe1981}), we have $\tI(\sigma) \cong \tI(\sigma^\theta)$. Artin $L$-functions and root numbers remain the same for equivalent Weil-Deligne representations, thus
\begin{equation*}
   \gamma(s,\tI(\sigma),\psi) = \gamma(s,\tI(\sigma^\theta),\psi).
\end{equation*}
Hence, by Theorem~\ref{HLAsai}, the first equation of the proposition follows.

The second equation, involving twisted Asai factors, follows from the first and the third. To prove equation~\eqref{asaiRS}, we first use multiplicativity of $\gamma$-factors to establish it for principal series representations. Then we treat the general case via the local-to-global technique of \cite{HeLo2013a, HeLo2013}.
\end{proof}

\begin{corollary}\label{RSAsaicor}
Let $\pi$ be an irreducible smooth representation of ${\rm GL}_n(l)$. Then
\begin{align*}
   L(s,\pi,r_\mathcal{A}) &= L(s,\pi^\theta,r_\mathcal{A}), \\
   \varepsilon(s,\pi,r_\mathcal{A},\psi) &= \varepsilon(s,\pi^\theta,r_\mathcal{A},\psi)
\end{align*}
and
\begin{align*}
   L(s,\pi \otimes \eta_{l/k},r_\mathcal{A}) &= L(s,\pi^\theta \otimes \eta_{l/k},r_\mathcal{A}), \\
   \varepsilon(s,\pi \otimes \eta_{l/k},r_\mathcal{A},\psi) &= \varepsilon(s,\pi^\theta \otimes \eta_{l/k},r_\mathcal{A},\psi).
\end{align*}
Furthermore, we have the following equation involving Rankin-Selberg and Asai factors
\begin{align*}
   L(s,\pi \times \pi^\theta) &= L(s,\pi,r_\mathcal{A}) L(s,\pi \otimes \eta_{l/k},r_\mathcal{A}) \\
   \varepsilon(s,\pi \times \pi^\theta,\psi_l) &= \varepsilon(s,\pi,r_\mathcal{A},\psi) \varepsilon(s,\pi \otimes \eta_{l/k},r_\mathcal{A},\psi).
\end{align*}
\end{corollary}
\begin{proof}
Since we are in the case of ${\rm GL}_n$, the argument found in \S~4.2 of \cite{HeLo2013} directly applies to the cases at hand.
\end{proof}

\section{Extended Langlands-Shahidi local factors for the unitary groups}\label{sec:extended}

Let ${\bf G}_1$ be either a quasi-split unitary group ${\rm U}_M$ or the group ${\rm Res}\,{\rm GL}_M$. In the case of ${\bf G}_1 = {\rm U}_M$, it is of rank $m$, where we write $M = 2m+1$ or $2m$ according to whether the unitary group is odd or even. Similarly, we let ${\bf G}_2$ be either a unitary group of rank $n$ or ${\rm Res}\,{\rm GL}_N$, with $N = 2n+1$ or $2n$. When convenient, we sometimes write $N_1$ for $M$ and $N_2$ for $N$.

In \S~\ref{separable} below we treat the case of the separable quadratic algebra $l = k \times k$; thus we extend the local theory to all degree-$2$ finite \'etale algebras $l$ over the non-archimedean local field $k$. We then proceed to list defining and basic properties for an extended system of LS $L$-functions, $\gamma$-factors and root numbers in the remaining sections.

\subsection{Local notation} Let $\mathfrak{ls}({\bf G}_1,{\bf G}_2,p)$ be the category whose objects are quadruples $(l/k,\pi,\tau,\psi)$ consisting of: a non-archimedean local field $k$, with ${\rm char}(k) = p$; a degree-$2$ finite \'etale algebra $l$ over $k$; irreducible admissible representations $\pi$ of $G_1 = {\bf G}_1(k)$ and $\tau$ of $G_2 = {\bf G}_2(k)$; and, a smooth non-trivial additive character $\psi : k \rightarrow \mathbb{C}^\times$. When ${\bf G}_1$ and ${\bf G}_2$ are clear from context, we will simply write $\mathfrak{ls}(p)$ for $\mathfrak{ls}({\bf G}_1,{\bf G}_2,p)$.

We obtain a character $\psi_l : l \rightarrow \mathbb{C}^\times$ from the character $\psi$ of $k$ via the trace, i.e., $\psi_l = \psi \circ {\rm Tr}_{l/k}$. As in \cite{LomLS}, generic representations are always assumed to be irreducible. We say that $(l/k,\pi,\tau,\psi) \in \mathfrak{ls}(p)$ is generic (resp. supercuspidal, discrete series, \ldots) if both $\pi$ and $\tau$ are generic (resp. supercuspidal, discrete series, \ldots) representations. Given a parabolic subgroup ${\bf P}$ of a connected reductive group ${\bf G}$ defined over $k$, we write ${\rm ind}_P^G$ for normalized parabolic induction. When $P$ and $G$ are clear from context, we simply write ${\rm Ind}$. For a non-archimedean local field $K$ in general, we let $q_K$ denote the cardinality of its residue field.

\subsection{Global notation} Let $\mathcal{LS}({\bf G}_1,{\bf G}_2,p)$, or simply $\mathcal{LS}(p)$, be the category whose objects are quintuples $(E/F,\pi,\tau,\psi,S)$ consisting of: a separable quadratic extension of global function fields $E/F$, with ${\rm char}(F) = p$; globally generic cuspidal automorphic representations $\pi = \otimes_v \pi_v$ of ${\bf G}_1(\mathbb{A}_F)$ and $\tau = \otimes_v \tau_v$ of ${\bf G}_2(\mathbb{A}_F)$; a non-trivial character $\psi = \otimes_v \psi_v : F \backslash \mathbb{A}_F \rightarrow \mathbb{C}^\times$; and, a finite set of places $S$ such that $E/F$, $\pi$, $\tau$ and $\psi$ are unramified outside of $S$.

We let $q$ be the cardinality of the field of constants of $F$. And, for every place $v$ of $F$, we let $q_v$ be the cardinality of the residue field of $F_v$.

Let $(E/F,\pi,\tau,\psi,S) \in \mathcal{LS}(p)$, then we have partial $L$-functions
\begin{equation*}
   L^S(s,\pi \times \tau) = \prod_{v \notin S} L(s,\pi_v \times \tau_v).
\end{equation*}
When the set $S$ plays no role, we omit it from the notation and simply write $(E/F,\pi,\tau,\psi) \in \mathcal{LS}(p)$.

\subsection{The case of a separable algebra}\label{separable} Let $l = k \times k$. We then have the following possibilities for Langlands-Shahidi $\gamma$-factors:
\begin{itemize}
   \item[(i)] Let $(l/k,\pi,\tau,\psi) \in \mathfrak{ls}({\rm U}_M,{\rm U}_N,p)$. Then $\pi$ is a representation of ${\rm GL}_M(k)$ and $\tau$ one of ${\rm GL}_N(k)$. The local functorial lift of $\pi$ to ${\bf H}_M(k)$ obtained from
   \begin{equation*}
      {\rm U}_M(k) = {\rm GL}_M(k) \rightsquigarrow {\bf H}_M(k) = {\rm GL}_M(k) \times {\rm GL}_M(k)
   \end{equation*}
   is given by $\pi \otimes \tilde{\pi}$. Similarly the local functorial lift of $\tau$ to ${\bf H}_N(k)$ obtained from
   \begin{equation*}
      {\rm U}_N(k) = {\rm GL}_N(k) \rightsquigarrow {\bf H}_N(k) = {\rm GL}_N(k) \times {\rm GL}_N(k)
   \end{equation*}
   is given by $\tau \otimes \tilde{\tau}$. Then the Langlands-Shahidi local factors are
   \begin{align*}
      \gamma_{l/k}(s,\pi \times \tau,\psi_l) &= \gamma(s,\pi \times \tau,\psi) \gamma(s,\tilde{\pi} \times \tilde{\tau},\psi) \\
      L_{l/k}(s,\pi \times \tau) &= L(s,\pi \times \tau) L(s,\tilde{\pi} \times \tilde{\tau}) \\
      \varepsilon_{l/k}(s,\pi \times \tau,\psi_l) &= \varepsilon(s,\pi \times \tau,\psi) \varepsilon(s,\tilde{\pi} \times \tilde{\tau},\psi).
   \end{align*}
   \item[(ii)] Let $(l/k,\pi,\tau,\psi) \in \mathfrak{ls}({\rm U}_M,{\rm Res}\,{\rm GL}_N,p)$. Then $\pi$ is a representation of ${\rm GL}_M(k)$ and $\tau$ one of ${\rm GL}_N(k) \times {\rm GL}_N(k)$. The local functorial lift of $\pi$ to ${\bf H}_M(k)$ obtained from
   \begin{equation*}
      {\rm U}_M(k) = {\rm GL}_M(k) \rightsquigarrow {\bf H}_M(k) = {\rm GL}_M(k) \times {\rm GL}_M(k)
   \end{equation*}
   is given by $\pi \otimes \tilde{\pi}$. Write $\tau = \tau_1 \otimes \tau_2$ as a representation of
   \begin{equation*}
      {\rm Res}_{l/k}{\rm GL}_N(k) = {\rm GL}_N(k) \times {\rm GL}_N(k).
   \end{equation*}
   Then the Langlands-Shahidi local factors are
   \begin{align*}
      \gamma_{l/k}(s,\pi \times \tau,\psi_l) &= \gamma(s,\pi \times \tau_1,\psi) \gamma(s,\tilde{\pi} \times \tau_2,\psi) \\
      L_{l/k}(s,\pi \times \tau) &= L(s,\pi \times \tau_1) L(s,\tilde{\pi} \times \tau_2) \\
      \varepsilon_{l/k}(s,\pi \times \tau,\psi_l) &= \varepsilon(s,\pi \times \tau_1,\psi) \varepsilon(s,\tilde{\pi} \times \tau_2,\psi).
   \end{align*}
   \item[(iii)] Let $(l/k,\pi,\tau,\psi) \in \mathfrak{ls}({\rm Res}\,{\rm GL}_M,{\rm Res}\,{\rm GL}_N,p)$. Then $\pi = \pi_1 \otimes \pi_2$ is a representation of ${\rm GL}_M(k) \times {\rm GL}_M(k)$ and $\tau = \tau_1 \otimes \tau_2$ one of ${\rm GL}_N(k) \times {\rm GL}_N(k)$. Then the Langlands-Shahidi local factors are
   \begin{align*}
      \gamma_{l/k}(s,\pi \times \tau,\psi_l) &= \gamma(s,\pi_1 \times \tau_1,\psi) \gamma(s,\pi_2 \times \tau_2,\psi) \\
      L_{l/k}(s,\pi \times \tau) &= L(s,\pi_1 \times \tau_1) L(s,\pi_2 \times \tau_2) \\
      \varepsilon_{l/k}(s,\pi \times \tau,\psi_l) &= \varepsilon(s,\pi_1 \times \tau_1,\psi) \varepsilon(s,\pi_2 \times \tau_2,\psi).
   \end{align*}
\end{itemize}

\begin{remark}
We usually drop the subscripts $l/k$ when dealing with Langlands-Shahidi local factors. Hopefully, it is clear from context what we mean by an $L$-function, and related local factors, at split places of a global function field.
\end{remark}

\begin{remark}
It is interesting to note that the theory of the Langlands-Shahidi local coefficient can be treated directly and uniformly for unitary groups defined over a degree-$2$ finite \'etale algebra $l$ over a non-archimedean local field $k$ as in \cite{Lo2016}. Alternatively, one can use the isomorphism ${\rm U}_N \cong {\rm GL}_N$ in the case of the split algebra $l = k \times k$.
\end{remark}

\subsection{Axioms}\label{axioms} We now provide a list of properties that uniquely characterize the rules $\gamma$, $L$ and $\varepsilon$ on $\mathfrak{ls}(p)$. Locally, we assume that we have quadratic field extensions $l/k$.
\begin{itemize}
   \item[(i)] (Naturality). \emph{Let $(l/k,\pi,\tau,\psi) \in \mathfrak{ls}(p)$ be generic and let $\eta : l'/k' \rightarrow l/k$ be an isomorphism of local field extensions. Let $(l'/k',\pi',\tau',\psi') \in \mathfrak{ls}(p)$ be the quadruple obtained via $\eta$. Then}
   \begin{equation*}
      \gamma(s,\pi \times \tau,\psi_l) = \gamma(s,\pi' \times \tau',\psi'_{l'}).
   \end{equation*}
   \item[(ii)] (Isomorphism). \emph{Let $(l/k,\pi,\tau,\psi)$, $(l/k,\pi',\tau',\psi) \in \mathfrak{ls}(p)$ be generic quadruples such that $\pi \cong \pi'$ and $\tau \cong \tau'$. Then}
   \begin{equation*}
      \gamma(s,\pi \times \tau,\psi_l) = \gamma(s,\pi' \times \tau',\psi_l).
   \end{equation*}
   \item[(iii)] (Compatibility with class field theory). \emph{Let ${\bf G}_i$ be either ${\rm U}_1$ or ${\rm Res}\,{\rm GL}_1$, for $i = 1$ or $2$, and let $(l/k,\chi_1,\chi_2,\psi) \in \mathfrak{ls}({\bf G}_1,{\bf G}_2,p)$. In the case of ${\rm U}_1$, we extend a character $\chi_i$ of ${\rm U}_1(k) = l^1$ to one of ${\rm Res}_{l/k}{\rm GL}_1(k) = {\rm GL}_1(l)$ via Hilbert's theorem 90. Then}
   \begin{equation*}
      \gamma(s,\chi_1 \times \chi_2,\psi_l) = \gamma(s,\chi_1\chi_2,\psi_l),
   \end{equation*}
   \emph{where the $\gamma$-factors on the right hand side are those of Tate's thesis for ${\rm GL}_1(l)$.}
   \item[(iv)] (Multiplicativity). \emph{Let $i = 1$ or $2$. For $1 \leq j \leq b_i$, let ${\bf G}_{i,j} = {\rm Res}\,{\rm GL}_{n_{i,j}}$, and take ${\bf G}_{i,0}$ of the same type as ${\bf G}_i$. Let $(l/k,\pi_{1,d},\pi_{2,e},\psi) \in \mathfrak{ls}({\bf G}_{1,d},{\bf G}_{2,e},p)$ be generic, for $0 \leq d \leq b_1$, $0 \leq e \leq b_2$. Suppose that each $\pi_i$ is the generic constituent of}
   \begin{equation*}
      {\rm ind}_{P_i}^{G_i}(\pi_{i,b_i} \otimes \cdots \otimes \pi_{i,1} \otimes \pi_{i,0}),
   \end{equation*}
   \emph{where ${\bf P}_i$ is the parabolic subgroup of ${\bf G}_i$ with Levi}
   \begin{equation*}
      {\bf M}_i = \prod_{j=1}^{b_i} {\rm Res}_{l/k}{\rm GL}_{n_{i,j}} \times {\bf G}_{i,0}.
   \end{equation*}
   \emph{Then}
   \begin{equation*}
      \gamma(s,\pi_1 \times \pi_2,\psi_l) = \prod_{d,e} \gamma(s,\rho_{1,d} \times \rho_{2,e},\psi_l),
   \end{equation*}
   \emph{where for each $i = 1$ or $2$, $\rho_{i,j}$ is the $j$-th entry of $\mathfrak{A}_i$ defined by}
   \begin{equation*}
      \mathfrak{A}_i = \left\{ \begin{array}{ll} (\pi_{i,b_i}, \ldots, \pi_{i,1}, \pi_{i,0}, \tilde{\pi}_{i,1}^\theta, \ldots, \tilde{\pi}_{i,b_i}^\theta) & \text{if } {\bf G}_i = {\rm U}_{N_i} \\ (\pi_{i,b_i}, \ldots, \pi_{i,1}, \pi_{i,0}) & \text{if } {\bf G}_i = {\rm Res}\,{\rm GL}_{N_i} \end{array} \right.
   \end{equation*}
   \item[(v)] (Dependence on $\psi$). \emph{Let $(l/k,\pi,\tau,\psi) \in \mathfrak{ls}(p)$ be generic and let $a \in k^\times$, and let $\psi_l^a$ be the character of $l$ defined by $\psi_l^a(x) = \psi_l(ax)$. Let $\omega_\pi$ and $\omega_\tau$ be the central characters of $\pi$ and $\tau$, extended to $l^\times$ via Hilbert's theorem 90 for unitary groups. Then}
   \begin{equation*}
      \gamma(s,\pi \times \tau,\psi_l^a) = \omega_\pi(a)^N \omega_\tau(a)^M \left| a \right|_l^{MN(s-\frac{1}{2})} \gamma(s,\pi \times \tau,\psi_l).
   \end{equation*}
   \item[(vi)] (Functional Equation). \emph{Let $(E/F,\pi,\tau,\psi,S) \in \mathcal{LS}(p)$, then}
   \begin{equation*}
      L^S(s,\pi \times \tau) = \prod_{v \in S} \gamma(s,\pi_v \times \tau_v,\psi_v) \, L^S(1-s,\tilde{\pi} \times \tilde{\tau}).
   \end{equation*}
   \emph{At split places $v$ of $F$, where $E_v \cong F_v \times F_v$, the Langlands-Shahidi local factors are the ones of \S~\ref{separable}.}
   \item[(vii)] (Tempered $L$-functions). \emph{For $(l/k,\pi,\tau,\psi) \in \mathfrak{ls}(p)$ tempered, let $P_{\pi \times \tau}(t)$ be the polynomial with $P_{\pi \times \tau}(0) = 1$ and such that $P_{\pi \times \tau}(q_k^{-s})$ is the numerator of $\gamma(s,\pi \times \tau,\psi_l)$. Then}
   \begin{equation*}
      L(s,\pi \times \tau) = \frac{1}{P_{\pi \times \tau}(q_k^{-s})}
   \end{equation*}
   \emph{is holomorphic and non-zero for $\Re(s) > 0$.}
   \item[(viii)] (Tempered $\varepsilon$-factors). \emph{Let $(l/k,\pi,\tau,\psi) \in \mathfrak{ls}(p)$ be tempered, then}
   \begin{equation*}
      \varepsilon(s,\pi \times \tau,\psi_l) = \gamma(s,\pi \times \tau,\psi_l) \frac{L(s,\pi \times \tau)}{L(1-s,\tilde{\pi} \times \tilde{\tau})}.
   \end{equation*}
   \item[(ix)] (Twists by unramified characters). \emph{Let $(l/k,\pi,\tau,\psi) \in \mathfrak{ls}({\rm U}_M,{\rm Res}\,{\rm GL}_N,p)$. Then}
   \begin{align*}
      L(s+s_0,\pi \times \tau) &= L(s,\pi \times (\tau \left| \det(\cdot) \right|_l^{s_0})), \\
      \varepsilon(s+s_0,\pi \times \tau,\psi_l) &= \varepsilon(s,\pi \times (\tau \left| \det(\cdot) \right|_l^{s_0}),\psi_l).
   \end{align*}
   \item[(x)] (Langlands classification). \emph{Let $i = 1$ or $2$. For $1 \leq d \leq b_1$ and $1 \leq e \leq b_2$, let $(l/k,\pi_{1,d},\pi_{2,e},\psi) \in \mathfrak{ls}({\rm Res}\,{\rm GL}_{n_{1,d}},{\rm Res}\,{\rm GL}_{n_{2,e}},p)$ be quasi-tempered. Let ${\bf G}_{i,0}$ be of the same type as ${\bf G}_i$, and take $(l/k,\pi_{1,0},\pi_{2,0},\psi) \in \mathfrak{ls}({\bf G}_{1,0},{\bf G}_{2,0},p)$ to be tempered generic. Suppose that each $\pi_i$ is the Langlands quotient of}
   \begin{equation*}
      {\rm ind}_{P_i}^{G_i}(\pi_{i,b_i} \otimes \cdots \otimes \pi_{i,1} \otimes \pi_{i,0})
   \end{equation*}
   \emph{where ${\bf P}_i$ is the parabolic subgroup of ${\bf G}_i$ with Levi}
   \begin{equation*}
      {\bf M}_i = \prod_{j=1}^{b_i} {\rm Res}_{l/k}{\rm GL}_{n_{i,j}} \times {\bf G}_{i,0}.
   \end{equation*}
   \emph{Furthermore, we can assume that each $\pi_{i,j}$, for $j \neq 0$, is of the form $\tau_{i,j} \left| \det(\cdot) \right|_l^{r_{i,j}}$ with each $\tau_{i,j}$ tempered and $0 < r_{i,1} < \cdots < r_{i,b_i}$. Then}
   \begin{align*}
      L(s,\pi_1 \times \pi_2) &= \prod_{d,e} L(s,\rho_{1,d} \times \rho_{2,e}) \\
      \varepsilon(s,\pi_1 \times \pi_2,\psi_l) &= \prod_{d,e} \varepsilon(s,\rho_{1,d} \times \rho_{2,e},\psi_l),
   \end{align*}
   \emph{where for each $i = 1$ or $2$, $\rho_{i,j}$ is the $j$-th component of $\mathfrak{A}_i$ defined by}
   \begin{equation*}
      \mathfrak{A}_i = \left\{ \begin{array}{ll} (\pi_{i,b_i}, \ldots, \pi_{i,1}, \pi_{i,0}, \tilde{\pi}_{i,1}^\theta, \ldots, \tilde{\pi}_{i,b_i}^\theta) & \text{if } {\bf G}_i = {\rm U}_{N_i} \\ (\pi_{i,b_i}, \ldots, \pi_{i,1}, \pi_{i,0}) & \text{if } {\bf G}_i = {\rm Res}\,{\rm GL}_{N_i} \end{array} \right.
   \end{equation*}
\end{itemize}

We observe that our main result, Theorem~\ref{mainthm}, includes existence and uniqueness of an extended system of factors in $\mathfrak{ls}(p)$, for generic representations. In particular, the above ten axioms uniquely characterize the system of $\gamma$-factors, $L$-functions and root numbers, as pointed out in the introduction; we omit the proof of uniqueness. The proof of existence for extended factors is completed in \S~\ref{sec:RR}.

We also comment on some of the individual axioms: in Property~(iii) we relate to the $\gamma$-factors of Tate's thesis \cite{Ta1950}; the case of a separable algebra, treated in \S~\ref{separable}, is used in Property~(vi); and, Property~(vii) is settled in \cite{dCHeLo}. Note that the corresponding list of properties for Rankin-Selberg products of general linear groups may be found in \cite{HeLo2013a}.

\subsection{Additional properties of \texorpdfstring{$L$}{L}-functions and local factors}\label{additional} We obtain a local functional equation for $\gamma$-factors which is proved using only Properties (i)--(vi), as in \S~4.2 of \cite{Lo2015}.
\begin{itemize}
   \item[(xi)] (Local functional equation). \emph{Let $(l/k,\pi,\tau,\psi) \in \mathfrak{ls}(p)$, then}
   \begin{equation*}
      \gamma(s,\pi \times \tau,\psi_l) \gamma(1-s,\tilde{\pi} \times \tilde{\tau},\overline{\psi}_l) = 1.
   \end{equation*}
\end{itemize}

We define automorphic $L$-functions and root numbers for $(E/F,\pi,\tau,\psi,S) \in \mathcal{LS}(p)$ by setting
\begin{equation*}
   L(s,\pi \times \tau) = \prod_v L(s,\pi_v \times \tau_v) \text{ and } \varepsilon(s,\pi \times \tau) = \prod_v \varepsilon(s,\pi_v \times \tau_v,\psi_v).
\end{equation*}
They satisfy a functional equation, whose proof is completed in \S~\ref{proofmain}.
\begin{itemize}
   \item[(xii)] (Global functional equation). \emph{Let $(E/F,\pi,\tau,\psi,S) \in \mathcal{LS}(p)$, then}
   \begin{equation*}
      L(s,\pi \times \tau) = \varepsilon(s,\pi \times \tau) L(1-s,\tilde{\pi} \times \tilde{\tau}).
   \end{equation*}
\end{itemize}

The following property is Theorem~5.1 of \cite{GaLo2018} adapted to the case of unitary groups.
\begin{itemize}
   \item[(xiii)] (Stability). \emph{Let $(l/k,\pi_i,\tau_i,\psi) \in \mathfrak{ls}({\rm U}_N,{\rm Res}\,{\rm GL}_m,p)$, for $i = 1$ or $2$, be generic and such that $\omega_{\pi_1} = \omega_{\pi_2}$ and $\omega_{\tau_1} = \omega_{\tau_2}$. If $\eta : l^\times \rightarrow \mathbb{C}^\times$ is highly ramified, then}
   \begin{equation*}
      \gamma(s,\pi_1 \times (\tau_1 \cdot \eta),\psi_l) = \gamma(s,\pi_2 \times (\tau_2 \cdot \eta),\psi_l).
   \end{equation*}
\end{itemize}
We note that we also have the corresponding stability properties for local $L$-functions and root numbers
\begin{align*}
   L(s,\pi_1 \times (\tau_1 \cdot \eta)) &= L(s,\pi_2 \times (\tau_2 \cdot \eta)), \\
   \varepsilon(s,\pi_1 \times (\tau_1 \cdot \eta),\psi_l) &= \varepsilon(s,\pi_2 \times (\tau_2 \cdot \eta),\psi_l).
\end{align*}
Stability for local $L$-functions is a result of Shahidi \cite{Sh2000}. Stability for $\varepsilon$-factors follows by combining stability for $\gamma$-factors and $L$-functions via Property~(viii) above for tempered representations. The general case then follows from the Langlands classification, Property~(x).

\subsection{Stable form of local factors} The following lemma is very useful when establishing global base change. It is proved as indicated in the previous subsection; in particular, see Property~(xiii) for $\gamma$-factors and the comments following Property~(xiii) concerning $L$-functions and $\varepsilon$-factors.

\begin{lemma}\label{stableform}
Let $(l/k,\pi,\tau,\psi) \in \mathfrak{ls}({\rm U}_N,{\rm Res}\,{\rm GL}_m,p)$ be generic. Consider a quadruple $(l/k,\Pi,T,\psi) \in \mathfrak{ls}({\rm Res}\,{\rm GL}_N,{\rm Res}\,{\rm GL}_m,p)$, with $\Pi$ and $T$ principal series, such that $\omega_\tau = \omega_T$ and $\omega_\Pi$ is the character of $l^\times$ obtained from $\omega_\pi$ of $l^1$ via Hilbert's theorem 90. Then, whenever $\eta : l^\times \rightarrow \mathbb{C}^\times$ is highly ramified, we have that
\begin{align*}
   L(s,\pi \times (\tau \cdot \eta)) &= L(s,\Pi \times (T \cdot \eta)), \\
   \varepsilon(s,\pi \times (\tau \cdot \eta),\psi_l) &= \varepsilon(s,\Pi \times (T \cdot \eta),\psi_l), \\
   \gamma(s,\pi \times (\tau \cdot \eta),\psi_l) &= \gamma(s,\Pi \times (T \cdot \eta),\psi_l).
\end{align*}
\end{lemma}

\section{The converse theorem and base change for the unitary groups}\label{sec:BC}

We combine the Langlands-Shahidi method with the converse theorem and establish a preliminary functorial lift known as ``weak'' base change for globally generic representations.

\subsection{The converse theorem} Let us recall the converse theorem of Cogdell and Piatetski-Shapiro \cite{CoPS1994}. Over function fields, a first version goes back to Piatetski-Shapiro \cite{PS1976}; see also Appendix~B of \cite{La2002}. The applications to functoriality and base change, however, use a variant that incorporates twists by a character, given in \S~2 of \cite{CKPSS2001} over number fields. In the function field case, this twisted variant is established by H.~del Castillo in \S~3 of \cite{dCa}, as part of his Ph.D. thesis; see also Theorem~8.1 of \cite{Lo2009}.

Fix a global function field $E$, a finite set of places $S$ of $E$, a Gr\"o{\ss}en\-charakter $\eta : E^\times \backslash \mathbb{A}_E^\times \rightarrow \mathbb{C}^\times$ and an integer $N \geq 2$. Let $\mathcal{T}(S;\eta)$ be the set consisting of representations $\tau = \tau_0 \cdot \eta$ of ${\rm GL}_n(\mathbb{A}_E)$ such that: $1 \leq n \leq N-1$; and, $\tau_0$ is a cuspidal automorphic representation with $\tau_{0,v}$ unramified for every $v \in S$.

Let $\pi$ of ${\bf G}_1(\mathbb{A}_F)$ and $\tau$ of ${\bf G}_2(\mathbb{A}_F)$ be admissible representations whose $L$-function $L(s,\pi \times \tau)$ converges on some right half plane. We say $L(s,\pi \times \tau)$ is \emph{nice} if the following properties are satisfied:
\begin{itemize}
   \item[(i)] $L(s,\pi \times \tau)$ and $L(s,\tilde{\pi} \times \tilde{\tau})$ are polynomials in $\left\{ q^{-s}, q^s \right\}$.
   \item[(ii)] $L(s,\pi \times \tau) = \varepsilon(s,\pi \times \tau) L(1-s,\tilde{\pi} \times \tilde{\tau})$.
\end{itemize}

We note that condition~(i) implies that $L(s,\pi \times \tau)$ and $L(s,\tilde{\pi} \times \tilde{\tau})$ extend to entire functions on the whole complex plane, bounded in vertical strips.

\begin{theorem}[Converse theorem]\label{converse}
Let $\Pi = \otimes \, \Pi_v$ be an irreducible admissible representation of ${\rm GL}_N(\mathbb{A}_E)$ whose central character $\omega_\Pi$ is a Gr\"o{\ss}en\-charakter and whose $L$-function $L(s,\Pi) = \prod_v L(s,\Pi_v)$ is absolutely convergent in some right half-plane. Suppose that for every $\tau \in \mathcal{T}(S;\eta)$ the $L$-function $L(s,\Pi \times \tau)$ is nice. Then, there exists an automorphic representation $\Pi'$ of ${\rm GL}_N(\mathbb{A}_E)$ such that $\Pi_v \cong \Pi_v'$ for every $v \notin S$ at which $\Pi_v$ is unramified.
\end{theorem}

\subsection{Base change for the unitary groups} Let $E/F$ be a separable quadratic extension of global function fields, and let $\mathbb{A}_F$ and $\mathbb{A}_E$ denote the ring of ad\`eles of $F$ and $E$, respectively. Throughout, we say that a place $v$ of $F$ is \emph{inert} if $E_v$ is a field; this includes the finitely many places where $E/F$ ramifies. By the \v{C}ebotar\"ev density theorem, the places of $F$ that split in $E$ and the inert places each have density $1/2$. Now, let ${\bf H}_N = {\rm Res}\,{\rm GL}_N$. The groups ${\rm U}_N$ and ${\bf H}_N$ are related via the following homomorphism of $L$-groups
\begin{equation}\label{BCmap}
   {\rm BC} : {}^L{\rm U}_N = {\rm GL}_N(\mathbb{C}) \rtimes \mathcal{W}_F \hookrightarrow {}^LH_N = {\rm GL}_N(\mathbb{C}) \times {\rm GL}_N(\mathbb{C}) \rtimes \mathcal{W}_F.
\end{equation}

We say that a globally generic cuspidal automorphic representation $\pi = \otimes' \pi_v$ of ${\rm U}_N(\mathbb{A}_F)$ has a base change lift $\Pi = \otimes' \Pi_v$ to ${\bf H}_N(\mathbb{A}_F) = {\rm GL}_N(\mathbb{A}_E)$, if at every place where $\pi_v$ is unramified, we have that
\begin{equation*}
   L(s,\pi_v) = L(s,\Pi_v).
\end{equation*}
This notion of a base change lift is sometimes referred to as a weak lift. The strong base change lift requires equality of $L$-functions and $\varepsilon$-factors at every place $v$ of $F$; we will establish the strong base change lift in \S~\ref{sec:strong}.

\begin{remark}
In order to be more precise, the base change map or Langlands functorial lift for the unitary groups obtained from \eqref{BCmap} is known as ``stable'' base change. There is also ``unstable'' base change. See for example the ``stable'' and ``labile'' base change discussion for ${\rm U}_2$ of \cite{Fl1982}.
\end{remark}

\subsection{Unramified base change}\label{unramifiedBC} Let $\pi = \otimes' \pi_v$ be a globally generic cuspidal automorphic representation of ${\rm U}_N(\mathbb{A}_F)$. Fix a place $v$ of $F$ that remains inert in $E$ and such that $\pi_v$ is unramified. Two unramified $L$-parameters $\phi_v : \mathcal{W}_{F_v} \rightarrow {}^L{\rm U}_N$ and $\Phi_v : \mathcal{W}_{F_v} \rightarrow {}^LH_N$ are connected via the homomorphism of $L$-groups
\begin{equation*}
\begin{tikzcd}[column sep=small]
   {}^L{\rm U}_N \arrow[rr] & & {}^LH_N \\
   & & \\
   & \mathcal{W}_{F_v} \arrow[uul, "\phi_v"] \arrow[uur, "\Phi_v"'] &
\end{tikzcd}
\end{equation*}
given by the base change map of \eqref{BCmap}.

Each $\pi_v$, being unramified, is of the form
\begin{equation}
   \pi_v \hookrightarrow \left\{ \begin{array}{ll} {\rm Ind}(\chi_{1,v} \otimes \cdots \otimes \chi_{n,v} \otimes \mu_v) & \text{if } N = 2n+1 \\ {\rm Ind}(\chi_{1,v} \otimes \cdots \otimes \chi_{n,v}) & \text{if } N = 2n \end{array} \right.,
\end{equation}
with $\chi_{1,v}$, \ldots, $\chi_{n,v}$ unramified characters of $E_v^\times$. Let $\varpi_v$ be a uniformizer and let
\begin{equation*}
   \alpha_{i,v} = \chi_{i,v}(\varpi_v), \ i = 1, \ldots, n.
\end{equation*}
Let ${\rm Frob}_v$ denote the Frobenius element of $\mathcal{W}_{F_v}$. We know that $\pi_v$ is parametrized by the conjugacy class in ${}^L{\rm U}_N$
\begin{equation*}
   \phi_v({\rm Frob}_v) = \left\{ \begin{array}{ll} {\rm diag}(\alpha_{1,v}^{\frac{1}{2}}, \ldots, \alpha_{n,v}^{\frac{1}{2}}, 1, \alpha_{n,v}^{-\frac{1}{2}}, \ldots, \alpha_{1,v}^{-\frac{1}{2}}) \rtimes w_{\theta,v} & \text{if } N = 2n+1 \\ {\rm diag}(\alpha_{1,v}^{\frac{1}{2}}, \ldots, \alpha_{n,v}^{\frac{1}{2}}, \alpha_{n,v}^{-\frac{1}{2}}, \ldots, \alpha_{1,v}^{-\frac{1}{2}}) \rtimes w_{\theta,v} & \text{if } N = 2n \end{array} \right..
\end{equation*}
Then, from the results of \cite{Mi2011}, the $L$-parameter $\Phi_v = {\rm BC} \circ \phi_v$ corresponds to a semisimple conjugacy class in ${\rm GL}_N(\mathbb{C})$ given by
\begin{equation*}
   \Phi_v({\rm Frob}_v) = \left\{ \begin{array}{ll} {\rm diag}(\alpha_{1,v}, \ldots, \alpha_{n,v}, 1, \alpha_{n,v}^{-1}, \ldots, \alpha_{1,v}^{-1}) & \text{if } N = 2n+1 \\ {\rm diag}(\alpha_{1,v}, \ldots, \alpha_{n,v}, \alpha_{n,v}^{-1}, \ldots, \alpha_{1,v}^{-1}) & \text{if } N = 2n \end{array} \right..
\end{equation*}
The resulting Satake parameter $\Phi_v$ then uniquely determines an unramified representation $\Pi_v$ of ${\rm Res}_{E_v/F_v}{\rm GL}_N(F_v) = {\rm GL}_N(E_v)$ of the form
\begin{equation}\label{unramifiedPi}
   \Pi_v \hookrightarrow \left\{ \begin{array}{ll} {\rm Ind}(\chi_{1,v} \otimes \cdots \otimes \chi_{n,v} \otimes 1 \otimes \chi_{n,v}^{-1} \otimes \cdots \otimes \chi_{1,v}^{-1}) & \text{if } N = 2n+1 \\ {\rm Ind}(\chi_{1,v} \otimes \cdots \otimes \chi_{n,v} \otimes \chi_{n,v}^{-1} \otimes \cdots \otimes \chi_{1,v}^{-1}) & \text{if } N = 2n \end{array} \right..
\end{equation}

To summarize, let $\hat{A}_v$ be the semisimple conjugacy class in ${\rm GL}_N(\mathbb{C})$ of the diagonal matrix in $\phi_v({\rm Frob}_v)$, obtained via the Satake parametrization. We have
\begin{equation*}
\begin{tikzcd}
   \pi_v \text{ of } {\rm U}_N(F_v) \arrow[r] \arrow[dd, dashed, "{\rm BC}"'] & \left\{ (\hat{A}_v, w_{\theta,v}) \right\} \text{ of } {\rm GL}_N(\mathbb{C}) \rtimes \mathcal{W}_{F_v} \arrow[dd, hook] \\
   & \\
   \Pi_v \text{ of } {\rm GL}_N(E_v) & \left\{ (\hat{A}_v, \hat{A}_v, w_{\theta,v}) \right\} \text{ of } {\rm GL}_N(\mathbb{C}) \times {\rm GL}_N(\mathbb{C}) \rtimes \mathcal{W}_{F_v} \arrow[l]
\end{tikzcd}
\end{equation*}
where we use the fact that there is a natural bijection between $w_{\theta,v}$-conjugacy classes of ${\rm GL}_N(\mathbb{C}) \times {\rm GL}_N(\mathbb{C})$ and conjugacy classes of ${\rm GL}_N(\mathbb{C})$.

\begin{definition}\label{unramifiedBCdef}
Let $v$ be a place of $F$ that remains inert in $E$. Let $\pi_v$ be an unramified representation with $L$-parameter $\phi_v$. We call the representation
\begin{equation*}
   {\rm BC}(\pi_v) = \Pi_v,
\end{equation*}
corresponding to $\Phi_v$ as in \eqref{unramifiedPi}, the unramified local Langlands lift or the unramified base change of $\pi_v$.
\end{definition}

Let $\Pi_v = {\rm BC}(\pi_v)$ be as in Definition~\ref{unramifiedBCdef}, and let $\tau_v$ be any generic representation of ${\rm GL}_m(E_v)$. Then
\begin{align*}
   \gamma(s,\pi_v \times \tau_v,\psi_v) &= \gamma(s,\Pi_v \times \tau_v,\psi_v) \\
   L(s,\pi_v \times \tau_v) &= L(s,\Pi_v \times \tau_v) \\
   \varepsilon(s,\pi_v \times \tau_v,\psi_v) &= \varepsilon(s,\Pi_v \times \tau_v,\psi_v).
\end{align*}

\subsection{Split base change}\label{splitBC} At split places $v$ of $F$, we are in the case of a separable algebra, as in \S~\ref{separable}. The local functorial lift of $\pi_v$ to ${\bf H}_N(F_v)$ is obtained from
\begin{equation*}
   {\rm U}_N(F_v) = {\rm GL}_N(F_v) \rightsquigarrow {\bf H}_N(F_v) = {\rm GL}_N(F_v) \times {\rm GL}_N(F_v).
\end{equation*}

\begin{definition}\label{splitBCdef}
Fix a place $v$ of $F$ such that $E_v = F_v \times F_v$. Let $\pi_v$ be a generic representation of ${\rm U}_N(F_v) \cong {\rm GL}_N(F_v)$. We call the representation
\begin{equation}
   {\rm BC}(\pi_v) = \pi_v \otimes \tilde{\pi}_v
\end{equation}
the split local Langlands lift or the split base change of $\pi_v$.
\end{definition}

Let $\tau_v = \tau_{1,v} \otimes \tau_{2,v}$ be any generic representation of ${\rm GL}_m(E_v) = {\rm GL}_m(F_v) \times {\rm GL}_m(F_v)$. Then, from \S~\ref{separable}, we have the following equality of local factors
\begin{align*}
   \gamma(s,\pi_v \times \tau_v,\psi_v) &= \gamma(s,\pi_v \times \tau_{1,v},\psi_v) \gamma(s,\tilde{\pi}_v \times \tau_{2,v},\psi_v) = \gamma(s,{\rm BC}(\pi_v) \times \tau_v,\psi_v) \\
   L(s,\pi_v \times \tau_v) &= L(s,\pi_v \times \tau_{1,v}) L(s,\tilde{\pi}_v \times \tau_{2,v}) = L(s,{\rm BC}(\pi_v) \times \tau_v) \\
   \varepsilon(s,\pi_v \times \tau_v,\psi_v) &= \varepsilon(s,\pi_v \times \tau_{1,v},\psi_v) \varepsilon(s,\tilde{\pi}_v \times \tau_{2,v},\psi_v) = \varepsilon(s,{\rm BC}(\pi_v) \times \tau_v,\psi_v).
\end{align*}

\subsection{Ramified base change}\label{ramifiedBC} At places $v$ of $F$ where the cuspidal automorphic representation $\pi$ of ${\rm U}_N(\mathbb{A}_F)$ may have ramification, we can use stability of $\gamma$-factors. At this point, we do not have a unique ramified base change, even after twisting by a highly ramified character. However, we will fully establish a unique local Langlands lift, or local base change, in \S~\ref{sec:LLF}.

\begin{definition}\label{ramifiedBCdef}
Let $\chi_{1,v}, \ldots, \chi_{n,v}$ be characters of $E_v^\times$. Let $\mu_v$ be a character of $E_v^1$, which we extend to one of $E_v^\times$ via Hilbert's theorem 90. Assume that the generic constituent
\begin{equation*}
   \Pi_v \hookrightarrow \left\{ \begin{array}{ll} {\rm Ind}(\chi_{1,v} \otimes \cdots \otimes \chi_{n,v} \otimes \mu_v \otimes \chi_{n,v}^{-1} \otimes \cdots \otimes \chi_{1,v}^{-1}) & \text{if } N = 2n+1 \\ {\rm Ind}(\chi_{1,v} \otimes \cdots \otimes \chi_{n,v} \otimes \chi_{n,v}^{-1} \otimes \cdots \otimes \chi_{1,v}^{-1}) & \text{if } N = 2n \end{array} \right.
\end{equation*}
has central character $\omega_{\Pi_v}$ obtained from $\omega_{\pi_v}$ via Hilbert's theorem 90. Then $\Pi_v$ is called a ramified local Langlands lift or a ramified base change of $\pi_v$.
\end{definition}

We no longer have equality of local factors for every $\tau_v$ of ${\rm GL}_m(E_v)$. However, from Lemma~\ref{stableform}, whenever $\eta_v : E_v^\times \rightarrow \mathbb{C}^\times$ is a highly ramified character, we have that
\begin{align*}
   \gamma(s,\pi_v \times (\tau_v \cdot \eta_v),\psi_v) &= \gamma(s,\Pi_v \times (\tau_v \cdot \eta_v),\psi_v) \\
   L(s,\pi_v \times (\tau_v \cdot \eta_v)) &= L(s,\Pi_v \times (\tau_v \cdot \eta_v)) \\
   \varepsilon(s,\pi_v \times (\tau_v \cdot \eta_v),\psi_v) &= \varepsilon(s,\Pi_v \times (\tau_v \cdot \eta_v),\psi_v).
\end{align*}

\subsection{Weak base change} We now arrive at the preliminary weak base change lift for the unitary groups.

\begin{proposition}\label{weakBC}
Let $\pi = \otimes' \pi_v$ be a globally generic cuspidal automorphic representation of ${\rm U}_N(\mathbb{A}_F)$. There exists a unique globally generic automorphic representation
\begin{equation*}
   {\rm BC}(\pi) = \Pi
\end{equation*}
of ${\rm Res}_{E/F}{\rm GL}_N(\mathbb{A}_F) = {\rm GL}_N(\mathbb{A}_E)$, which is a weak base change lift of $\pi$.
\end{proposition}
\begin{proof}
Let $\Pi_v = {\rm BC}(\pi_v)$ be the local base change of Definitions~\ref{unramifiedBCdef}, \ref{splitBCdef} and \ref{ramifiedBCdef}, accordingly. Consider the irreducible admissible representation
\begin{equation*}
   \Pi^\flat = \otimes' \, \Pi_v
\end{equation*}
of ${\rm GL}_N(\mathbb{A}_E)$ whose central character $\omega_{\Pi^\flat}$ has $\omega_{\Pi_v}$ obtained from $\omega_{\pi_v}$ via Hilbert's theorem 90 at every place $v$ of $F$. By construction, $\omega_{\Pi^\flat}$ is invariant under $E^\times$.

Fix a non-trivial character $\psi = \otimes \psi_v : F \backslash \mathbb{A}_F \rightarrow \mathbb{C}^\times$. Let $S$ be a non-empty finite set of places of $F$ such that $E/F$, $\pi$ and $\psi$ are unramified outside of $S$. We abuse notation and identify $S$ with the finite set of places of $E$ lying above the places $v \in S$. Then, we have an equality of partial $L$-functions
\begin{equation*}
   L^S(s,\Pi^\flat) = L^S(s,\pi \times 1).
\end{equation*}
Hence, $L^S(s,\Pi^\flat)$ converges absolutely on a right half plane; and so does $L(s,\Pi^\flat)$.

Choose a Gr\"o{\ss}en\-charakter $\eta = \otimes \eta_v : E^\times \backslash \mathbb{A}_E^\times \rightarrow \mathbb{C}^\times$ such that $\eta_v$ is highly ramified for every $v \in S$, as in Lemma~\ref{stableform} and Proposition~4.1 of \cite{Lom19}. Let $\tau$ be a cuspidal automorphic representation of ${\rm GL}_m(\mathbb{A}_E)$, $1 \leq m \leq N-1$, such that $\tau_v$ is unramified for every $v \in S$. Then, letting $\tau' = \tau \cdot \eta \in \mathcal{T}(S;\eta)$ and taking a finite set of places $S_\tau \supseteq S$ of $F$ such that $\tau'$ is unramified outside of $S_\tau$, we have that $(E/F,\pi,\tau',\psi,S_\tau) \in \mathcal{LS}({\rm U}_N,{\rm Res}\,{\rm GL}_m,p)$. We have seen in \S\S~\ref{unramifiedBC}, \ref{splitBC} and \ref{ramifiedBC} that the following equality of local factors holds in every case:
\begin{align*}
   L(s,\pi_v \times \tau_v') &= L(s,\Pi_v \times \tau_v') \\
   \varepsilon(s,\pi_v \times \tau_v',\psi_v) &= \varepsilon(s,\Pi_v \times \tau_v',\psi_v).
\end{align*}
With this choice of $\eta$, we know that the Langlands-Shahidi $L$-functions $L(s,\pi \times \tau')$ are polynomials in $\left\{ q^s, q^{-s} \right\}$. They also satisfy the global functional equation, Property~(vi) of \S~\ref{axioms}. Thus, they are nice. Then, since
\begin{equation*}
   L(s,\Pi^\flat \times \tau') = L(s,\pi \times \tau') \text{ and } \varepsilon(s,\Pi^\flat \times \tau') = \varepsilon(s,\pi \times \tau'),
\end{equation*}
we can conclude that the $L$-functions $L(s,\Pi^\flat \times \tau')$ are nice, as $\tau'$ ranges through the set $\mathcal{T}(S;\eta)$. From the converse theorem (Theorem~\ref{converse}), there exists an automorphic representation $\Pi'$ of ${\rm GL}_N(\mathbb{A}_E)$ such that $\Pi_v \cong \Pi_v'$ for all $v \notin S$. Then $\Pi'$ gives a weak base change.

Now, from \cite{La1979}, the automorphic representation $\Pi'$ arises as a subquotient of a globally induced representation
\begin{equation}\label{globalinduced}
   {\rm Ind}(\Pi_1 \otimes \cdots \otimes \Pi_d),
\end{equation}
with each $\Pi_i$ a cuspidal automorphic representation of ${\rm GL}_{n_i}(\mathbb{A}_E)$, $N = n_1 + \cdots + n_d$. Since every $\Pi_i$ is cuspidal, they are globally generic. The results on the classification of automorphic representations for general linear groups \cite{JaSh1981} show that there exists a unique generic subquotient of \eqref{globalinduced}, which we denote (abusing notation) by
\begin{equation*}
   \Pi = \Pi_1 \boxplus \cdots \boxplus \Pi_d.
\end{equation*}
It is this automorphic representation $\Pi$ which is our desired base change, i.e.,
\begin{equation*}
   {\rm BC}(\pi) = \Pi.
\end{equation*}
Indeed, it has the property that at every place $w$ of $E$ where $\Pi_w$ is unramified, it is generic. Hence, at places where $\pi_v$ is unramified and $w = v$ remains inert, $\Pi_w$ is given by the unique generic subquotient of a principal series representation and $\Pi_w$ agrees with the local base change lift of \S~\ref{unramifiedBC}. At split places it also agrees with that of \S~\ref{splitBC} at almost all places. It thus agrees with $\Pi'$ at almost all places and is itself a weak base change lift. Furthermore, by strong multiplicity one \cite{PS1979} and its extension to isobaric representations \cite{JaSh1981}, any two globally generic automorphic representations of ${\rm GL}_N(\mathbb{A}_E)$ that agree at almost every place are equal. Hence $\Pi$ is uniquely determined.
\end{proof}

\section{Uniqueness of local base change and a global to local result}\label{sec:local}

In this section, we define local base change and show that it is unique when it exists; we also prove that global base change is compatible with local base change. Throughout, we fix a quadratic extension $l/k$ of non-archimedean local fields of positive characteristic. Given any general linear group ${\rm GL}_m(l)$, we let $\nu$ denote the unramified character obtained via the determinant, i.e., $\nu = \left| \det(\cdot) \right|_l$. Globally, we let $E/F$ denote a separable quadratic extension of function fields.

\begin{definition}\label{localBCdef}
Let $\pi$ be a generic representation of ${\rm U}_N(k)$. Then, we say that a generic representation $\Pi$ of ${\rm GL}_N(l)$ is \emph{a} local base change lift \emph{of} $\pi$ if for every supercuspidal representation $\tau$ of ${\rm GL}_m(l)$ we have that
\begin{equation*}
   \gamma(s,\pi \times \tau,\psi_l) = \gamma(s,\Pi \times \tau,\psi_l).
\end{equation*}
\end{definition}

We also refer to such a $\Pi$ as a local Langlands lift, or local Langlands functorial lift, of $\pi$; by Lemma~\ref{uniqueBC} below, it is unique when it exists.

\subsection{Uniqueness of the local base change lift} The equality of $\gamma$-factors in the previous definition extends to twists by a general unitary generic representation $\tau$ of ${\rm GL}_m(l)$, as we show in the next lemma. For this, the classification of \cite{Ta1986} is very useful. It allows us to write
\begin{equation}\label{tadic}
   \tau = {\rm Ind}(\delta_1\nu^{\lambda_1} \otimes \cdots \otimes \delta_d\nu^{\lambda_d} \otimes \delta_{d+1} \otimes \cdots \otimes \delta_{d+c} \otimes \delta_d\nu^{-\lambda_d} \otimes \cdots \otimes \delta_1\nu^{-\lambda_1}),
\end{equation}
where the $\delta_i$'s are unitary discrete series representations of ${\rm GL}_{n_i}(l)$ and $0 < \lambda_d \leq \cdots \leq \lambda_1 < 1/2$.

Furthermore, from the Zelevinsky classification \cite{Ze1980}, we know that every unitary discrete series representation $\delta$ of ${\rm GL}_m(l)$ is obtained from a segment of the form
\begin{equation*}
   \Delta = \left[ \rho\nu^{-\frac{t-1}{2}}, \rho\nu^{\frac{t-1}{2}} \right],
\end{equation*}
where $\rho$ is a supercuspidal representation of ${\rm GL}_e(l)$, $e|m$, and $t$ is a positive integer. The representation $\delta$ is precisely the generic constituent of
\begin{equation}\label{segment}
   {\rm Ind}(\rho\nu^{-\frac{t-1}{2}} \otimes \cdots \otimes \rho\nu^{\frac{t-1}{2}}).
\end{equation}

An important result of Henniart \cite{He1993} allows us to uniquely characterize the local Langlands functorial lift by the condition that it preserves local factors.

\begin{lemma}\label{uniqueBC}
Let $\pi$ be a generic representation of ${\rm U}_N(k)$ and suppose there exists $\Pi$, a local base change lift to ${\rm GL}_N(l)$. Then, for every unitary generic representation $\tau$ of ${\rm GL}_m(l)$ we have that
\begin{equation*}
   \gamma(s,\pi \times \tau,\psi_l) = \gamma(s,\Pi \times \tau,\psi_l).
\end{equation*}
Furthermore, such a local base change lift $\Pi$ is unique.
\end{lemma}
\begin{proof}
Given a unitary generic representation $\tau$ of ${\rm GL}_m(l)$, write $\tau$ in the form given by \eqref{tadic}. Multiplicativity of $\gamma$-factors, i.e., \S~\ref{axioms} Property~(iv), gives
\begin{align*}
   \gamma(s,\pi \times \tau,\psi_l) = & \prod_{i=1}^c \gamma(s,\pi \times \delta_{d+i},\psi_l) \\
   & \prod_{j=1}^d \gamma(s+\lambda_j,\pi \times \delta_j,\psi_l) \gamma(s-\lambda_j,\pi \times \delta_j,\psi_l).
\end{align*}
Similarly, multiplicativity for Rankin-Selberg $\gamma$-factors \cite{JPSS1983} (see also \cite{HeLo2013a}) gives
\begin{align*}
   \gamma(s,\Pi \times \tau,\psi_l) = & \prod_{i=1}^c \gamma(s,\Pi \times \delta_{d+i},\psi_l) \\
   & \prod_{j=1}^d \gamma(s+\lambda_j,\Pi \times \delta_j,\psi_l) \gamma(s-\lambda_j,\Pi \times \delta_j,\psi_l).
\end{align*}
In this way, we reduce the problem to proving the relation
\begin{equation*}
   \gamma(s,\pi \times \delta,\psi_l) = \gamma(s,\Pi \times \delta,\psi_l)
\end{equation*}
for discrete series representations $\delta$ of ${\rm GL}_m(l)$.

Now, we write the representation $\delta$ as the generic constituent of
\begin{equation*}
   {\rm Ind}(\rho\nu^{-\frac{t-1}{2}} \otimes \cdots \otimes \rho\nu^{\frac{t-1}{2}}),
\end{equation*}
as in \eqref{segment}. Then, using the multiplicativity property of $\gamma$-factors, we obtain
\begin{align*}
   \gamma(s,\pi \times \delta,\psi_l) &= \prod_{a=0}^{t-1} \gamma(s - \frac{t-1}{2} + a,\pi \times \rho,\psi_l) \\
   &= \prod_{a=0}^{t-1} \gamma(s - \frac{t-1}{2} + a,\Pi \times \rho,\psi_l) \\
   &= \gamma(s,\Pi \times \delta,\psi_l).
\end{align*}
This shows that $\Pi$ satisfies the desired relation involving $\gamma$-factors. That $\Pi$ is unique then follows from Theorem~1.1 of \cite{He1993}.
\end{proof}

\begin{lemma}\label{localLfactors}
Let $\pi$ be a generic representation of ${\rm U}_N(k)$ and suppose it has a unitary local Langlands functorial lift $\Pi$ of ${\rm GL}_N(l)$. Then, for every unitary generic representation $\tau$ of ${\rm GL}_m(l)$ we have that
\begin{align*}
   L(s,\pi \times \tau) &= L(s,\Pi \times \tau) \\
   \varepsilon(s,\pi \times \tau,\psi_l) &= \varepsilon(s,\Pi \times \tau,\psi_l).
\end{align*}
\end{lemma}
\begin{proof}
As in the previous lemma, begin by writing the unitary generic representation $\tau$ of ${\rm GL}_m(l)$ as in \eqref{tadic}, with discrete series as inducing data, and write
\begin{equation*}
   \tau_{\rm temp} = {\rm Ind}(\delta_{d+1} \otimes \cdots \otimes \delta_{d+c}).
\end{equation*}
Then, using Properties~(ix) and (x) of \S~\ref{axioms}, we obtain
\begin{align*}
   L(s,\pi \times \tau) &= L(s,\pi \times \tau_{\rm temp}) \prod_{i=1}^d L(s+\lambda_i,\pi \times \delta_i) L(s-\lambda_i,\pi \times \delta_i), \\
   \varepsilon(s,\pi \times \tau,\psi_l) &= \varepsilon(s,\pi \times \tau_{\rm temp},\psi_l) \prod_{i=1}^d \varepsilon(s+\lambda_i,\pi \times \delta_i,\psi_l) \varepsilon(s-\lambda_i,\pi \times \delta_i,\psi_l).
\end{align*}
Now, for the factors involving $\tau_{\rm temp}$, we use the Langlands classification to express $\pi$ as a Langlands quotient of
\begin{equation*}
   {\rm Ind}(\pi_e \otimes \cdots \otimes \pi_1 \otimes \pi_{\rm temp}).
\end{equation*}
Here, the induction is from the parabolic subgroup ${\bf P}$ of ${\rm U}_N$ with Levi ${\bf M} = \prod_{i=1}^e {\rm Res}\,{\rm GL}_{m_i} \times {\rm U}_{N_0}$, each $\pi_i$ is a quasi-tempered representation of ${\rm GL}_{m_i}(l)$, and $\pi_{\rm temp}$ is a tempered representation of ${\rm U}_{N_0}(k)$. Since $\pi$ is generic, it is also the generic constituent of this induced representation, by the standard module conjecture \cite{dCHeLo}. Now, Properties~(ix) and (x) of \S~\ref{axioms} in this situation directly give
\begin{align*}
   L(s,\pi \times \tau_{\rm temp}) &= L(s,\pi_{\rm temp} \times \tau_{\rm temp}) \prod_{i=1}^e L(s,\pi_i \times \tau_{\rm temp}), \\
   \varepsilon(s,\pi \times \tau_{\rm temp},\psi_l) &= \varepsilon(s,\pi_{\rm temp} \times \tau_{\rm temp},\psi_l) \prod_{i=1}^e \varepsilon(s,\pi_i \times \tau_{\rm temp},\psi_l).
\end{align*}
All representations involved in the previous two equations are quasi-tempered. Each individual local factor on the RHS of these equations can be shifted by Property~(ix) of \S~\ref{axioms}, which leads to an $L$-function and $\varepsilon$-factor involving tempered representations. The connection to $\gamma$-factors, and the previous lemma, is now made via Property~(viii) of \S~\ref{axioms} in addition to Property~(xi) of \S~\ref{additional}. Now, multiplicativity of $\gamma$-factors leads to
\begin{align*}
   L(s,\pi \times \tau_{\rm temp}) &= \prod_{a=1}^c L(s,\pi \times \delta_{d+a}), \\
   \varepsilon(s,\pi \times \tau_{\rm temp},\psi_l) &= \prod_{a=1}^c \varepsilon(s,\pi \times \delta_{d+a},\psi_l).
\end{align*}
The properties of \cite{HeLo2013a}, for example, can be applied to Rankin-Selberg factors to obtain
\begin{align*}
   L(s,\Pi \times \tau) &= \prod_{a=1}^c L(s,\Pi \times \delta_{d+a}) \prod_{i=1}^d L(s+\lambda_i,\Pi \times \delta_i) L(s-\lambda_i,\Pi \times \delta_i), \\
   \varepsilon(s,\Pi \times \tau,\psi_l) &= \prod_{a=1}^c \varepsilon(s,\Pi \times \delta_{d+a},\psi_l) \prod_{i=1}^d \varepsilon(s+\lambda_i,\Pi \times \delta_i,\psi_l) \varepsilon(s-\lambda_i,\Pi \times \delta_i,\psi_l),
\end{align*}
where we reduced to proving
\begin{align*}
   L(s,\pi \times \delta) &= L(s,\Pi \times \delta) \\
   \varepsilon(s,\pi \times \delta,\psi_l) &= \varepsilon(s,\Pi \times \delta,\psi_l)
\end{align*}
for discrete series representations $\delta$ of ${\rm GL}_m(l)$. Indeed, we now address this case in what follows.

For the unitary generic representation $\Pi$ of ${\rm GL}_N(l)$, we use \eqref{tadic} to write
\begin{equation}\label{tadicPi}
   \Pi = {\rm Ind}(\xi_1\nu^{r_1} \otimes \cdots \otimes \xi_f\nu^{r_f} \otimes \xi_{f+1} \otimes \cdots \otimes \xi_{f+h} \otimes \xi_f\nu^{-r_f} \otimes \cdots \otimes \xi_1\nu^{-r_1}),
\end{equation}
with each $\xi_i$ a discrete series and $0 < r_f \leq \cdots \leq r_1 < 1/2$. Then, by Lemma~\ref{uniqueBC} and multiplicativity of Rankin-Selberg $\gamma$-factors, we have
\begin{align*}
   \gamma(s,\pi \times \delta,\psi_l) = & \prod_{i=1}^h \gamma(s,\xi_{f+i} \times \delta,\psi_l) \\
   & \prod_{j=1}^f \gamma(s+r_j,\xi_j \times \delta,\psi_l) \gamma(s-r_j,\xi_j \times \delta,\psi_l).
\end{align*}
Each $\gamma$-factor on the right hand side of the previous expression involves discrete series (hence tempered) representations. Thus, each factor has a corresponding $L$-function and root number via Property~(viii) of \S~\ref{axioms}. We define the polynomial $P(t)$, with $P(0) = 1$, by
\begin{equation*}
   P(q_k^{-s})^{-1} = \prod_{i=1}^h L(s,\xi_{f+i} \times \delta) \prod_{j=1}^f L(s+r_j,\xi_j \times \delta) L(s-r_j,\xi_j \times \delta).
\end{equation*}
From the proposition on p.~451 of \cite{JPSS1983}, each $L(s,\xi_i \times \delta)$ has no poles for $\Re(s) > 0$. And since $r_j < 1/2$, the function $P(q_k^{-s})$ is non-zero for $\Re(s) \geq 1/2$. Now, the product
\begin{equation*}
   Q(q_k^{-s})^{-1} = \prod_{i=1}^h L(1-s,\tilde{\xi}_{f+i} \times \tilde{\delta}) \prod_{j=1}^f L(1-s-r_j,\tilde{\xi}_j \times \tilde{\delta}) L(1-s+r_j,\tilde{\xi}_j \times \tilde{\delta})
\end{equation*}
is in turn non-zero for $\Re(s) \leq 1/2$. Then, Property~(viii) of \S~\ref{axioms} gives the relation
\begin{equation*}
   \gamma(s,\pi \times \delta,\psi_l) \mathrel{\ddot{\sim}} \frac{P(q_k^{-s})}{Q(q_k^{-s})},
\end{equation*}
which is an equality up to a monomial in $q_k^{-s}$. More precisely, the monomial is the root number, which we can decompose as
\begin{align*}
   \varepsilon(s,\pi \times \delta,\psi_l) = & \prod_{i=1}^h \varepsilon(s,\xi_{f+i} \times \delta,\psi_l) \\
   & \prod_{j=1}^f \varepsilon(s+r_j,\xi_j \times \delta,\psi_l) \varepsilon(s-r_j,\xi_j \times \delta,\psi_l) \\
   = & \varepsilon(s,\Pi \times \delta,\psi_l).
\end{align*}

Notice that the regions where $P(q_k^{-s})$ and $Q(q_k^{-s})$ may be zero do not intersect. Hence, there are no cancellations involving the numerator and denominator of $\gamma(s,\pi \times \delta,\psi_l)$. This shows that
\begin{equation*}
   L(s,\pi \times \delta) = \frac{1}{P(q_k^{-s})} = L(s,\Pi \times \delta),
\end{equation*}
where the second equality follows using the form \eqref{tadicPi} of $\Pi$ and the multiplicativity property of Rankin-Selberg $L$-functions.
\end{proof}

Unitarity of the lift $\Pi$ is the hypothesis that allows us to use the classification of \cite{Ta1986} in Lemma~\ref{localLfactors}. However, we do not deduce temperedness of the lift from its unitarity. Instead, the lift of a generic supercuspidal representation is shown to be tempered by a purely local argument, Proposition~\ref{scTempered}. The lifts of generic discrete series and tempered representations of Propositions~\ref{discreteseries} and \ref{temperedBC} are then tempered by Theorem~\ref{temperedtempered}. Since tempered representations of ${\rm GL}_N(l)$ are unitary, Lemma~\ref{localLfactors} applies to all of these lifts. Unitarity of the local components of the global lift is a consequence of Theorem~\ref{strongBC}.

\begin{proposition}\label{scTempered}
Let $\pi_0$ be a generic supercuspidal representation of ${\rm U}_N(k)$, and let $\Pi_0$ be a local base change lift of $\pi_0$. Then
\begin{equation*}
   \Pi_0 \cong {\rm Ind}(\rho_1 \otimes \cdots \otimes \rho_d),
\end{equation*}
with each $\rho_i$ a unitary supercuspidal representation of ${\rm GL}_{m_i}(l)$. In particular, $\Pi_0$ is tempered, hence unitary.
\end{proposition}
\begin{proof}
By Definition~\ref{localBCdef}, we have in particular
\begin{equation*}
   \gamma(s,\pi_0 \times \rho,\psi_l) = \gamma(s,\Pi_0 \times \rho,\psi_l)
\end{equation*}
for every unitary supercuspidal representation $\rho$ of ${\rm GL}_m(l)$, which is the only form of the hypothesis that we use. By the Zelevinsky classification \cite{Ze1980}, $\Pi_0$ is the generic constituent of an induced representation
\begin{equation*}
   {\rm Ind}(\rho_1\nu^{\lambda_1} \otimes \cdots \otimes \rho_d\nu^{\lambda_d}),
\end{equation*}
with each $\rho_i$ a unitary supercuspidal representation of ${\rm GL}_{m_i}(l)$ and $\lambda_i \in \mathbb{C}$. We show that $\Re(\lambda_i) = 0$ for all $i$.

Fix $i_0$ and let $\rho = \rho_{i_0}$. Let $f$ be the number of unramified characters $\chi$ of $l^\times$ such that $\rho\chi \cong \rho$. Let $\Lambda$ be the multiset of those $\lambda_i$ for which $\rho_i \cong \rho\nu^{iy}$ for some $y \in \mathbb{R}$; after absorbing $iy$ into $\lambda_i$, we may assume that $\rho_i = \rho$ for $\lambda_i \in \Lambda$. For a unitary supercuspidal representation $\rho'$, we have that $L(s,\rho' \times \tilde{\rho}) = 1$ unless $\rho' \cong \rho\nu^{iy}$ for some $y \in \mathbb{R}$, and $L(s,\rho \times \tilde{\rho}) = (1 - q_l^{-fs})^{-1}$ \cite{JPSS1983}. Multiplicativity of Rankin-Selberg $\gamma$-factors \cite{JPSS1983, HeLo2013a} then gives
\begin{equation*}
   \gamma(s,\Pi_0 \times \tilde{\rho},\psi_l) \mathrel{\ddot{\sim}} \prod_{\lambda \in \Lambda} \frac{1 - q_l^{-f(s+\lambda)}}{1 - q_l^{-f(1-s-\lambda)}},
\end{equation*}
an equality up to a monomial in $q_k^{-s}$, as in the proof of Lemma~\ref{localLfactors}. The factor attached to $\lambda$ vanishes exactly at $s \in -\lambda + \frac{2\pi i}{f\log q_l}\mathbb{Z}$, and it has poles exactly at $s \in 1-\lambda + \frac{2\pi i}{f\log q_l}\mathbb{Z}$.

Let $\lambda^+ \in \Lambda$ have maximal real part. The denominator can vanish at $s = -\lambda^+$ only if there is a $\lambda \in \Lambda$ with $\Re(\lambda) = \Re(\lambda^+) + 1$, which is impossible. Hence $s = -\lambda^+$ is a zero of $\gamma(s,\Pi_0 \times \tilde{\rho},\psi_l)$. Similarly, let $\lambda^- \in \Lambda$ have minimal real part. The numerator can vanish at $s = 1-\lambda^-$ only if there is a $\lambda \in \Lambda$ with $\Re(\lambda) = \Re(\lambda^-) - 1$. Hence $s = 1-\lambda^-$ is a pole of $\gamma(s,\Pi_0 \times \tilde{\rho},\psi_l)$.

On the other hand, $\pi_0$ and $\tilde{\rho}$ are tempered, and Properties~(vii) and (viii) of \S~\ref{axioms} give
\begin{equation*}
   \gamma(s,\pi_0 \times \tilde{\rho},\psi_l) = \varepsilon(s,\pi_0 \times \tilde{\rho},\psi_l) \frac{L(1-s,\tilde{\pi}_0 \times \rho)}{L(s,\pi_0 \times \tilde{\rho})},
\end{equation*}
with $\varepsilon(s,\pi_0 \times \tilde{\rho},\psi_l)$ a monomial in $q_k^{-s}$. Since $\pi_0$ and $\tilde{\pi}_0$ are generic supercuspidal representations, the poles of $L(s,\pi_0 \times \tilde{\rho})$ and of $L(s,\tilde{\pi}_0 \times \rho)$ lie on the line $\Re(s) = 0$, see \S~5 of \cite{LomLS}; in characteristic zero, see \S~7 of \cite{Sh1990}. Hence the zeros of $\gamma(s,\pi_0 \times \tilde{\rho},\psi_l)$ lie on the line $\Re(s) = 0$ and its poles on the line $\Re(s) = 1$; in particular, there are no cancellations between numerator and denominator. Comparing with the previous paragraph, we obtain $\Re(\lambda^+) = 0 = \Re(\lambda^-)$. Hence $\Re(\lambda) = 0$ for every $\lambda \in \Lambda$. Since $i_0$ was arbitrary, every $\rho_i\nu^{\lambda_i}$ is unitary.

Finally, write $\Pi_0$ as the generic constituent of a product of essentially square integrable representations attached to segments, as in \eqref{segment}. A segment $\left[ \rho'\nu^{a}, \rho'\nu^{a+t-1} \right]$ all of whose elements are unitary must have $t = 1$. Hence
\[ \Pi_0 \cong {\rm Ind}(\rho_1\nu^{\lambda_1} \otimes \cdots \otimes \rho_d\nu^{\lambda_d}), \]
which is irreducible and tempered by \cite{Ze1980}; notice that each $\rho_i\nu^{\lambda_i}$ is unitary supercuspidal.
\end{proof}

\subsection{A global to local result} We now show that every local component of the weak base change lift of Proposition~\ref{weakBC} is the local base change of the corresponding local component of $\pi$. At unramified and split places, the local components of the global lift are those of \S\S~\ref{unramifiedBC} and \ref{splitBC}, and the central character is obtained via \eqref{hilbert90}.

\begin{proposition}\label{globallocal}
Given a globally generic cuspidal automorphic representation $\pi = \otimes' \pi_v$ of ${\rm U}_N(\mathbb{A}_F)$, let ${\rm BC}(\pi) = \Pi = \otimes' \Pi_v$ be the base change lift of Proposition~\ref{weakBC}. Then, for every $v$ we have:
\begin{itemize}
   \item[(i)] $\Pi_v$ is the uniquely determined local base change of $\pi_v$;
   \item[(ii)] $\Pi_v$ has central character $\omega_{\Pi_v}$ of $E_v^\times$ obtained from the character $\omega_{\pi_v}$ of $E_v^1$ via Hilbert's theorem 90.
\end{itemize}
\end{proposition}
\begin{proof}
The base change lift $\Pi = {\rm BC}(\pi)$, being globally generic, has every local $\Pi_v$ generic. At every place where $\pi_v$ is unramified, $\Pi_v$ is a constituent of an unramified principal series representation by construction. Since $\Pi_v$ is generic, it has to be the unique generic constituent of the principal series representation. Thus
\begin{equation}\label{unrBCv}
   \Pi_v = {\rm BC}(\pi_v)
\end{equation}
is the unramified local base change of \S~\ref{unramifiedBC}, if $v$ is inert, and that of \S~\ref{splitBC}, if $v$ is split.

Fix a place $v_0$ of $F$ which remains inert in $E$. We wish to show that
\begin{equation*}
   \gamma(s,\pi_{v_0} \times \tau_0,\psi_{v_0}) = \gamma(s,\Pi_{v_0} \times \tau_0,\psi_{v_0})
\end{equation*}
for every generic $(E_{v_0}/F_{v_0},\pi_{v_0},\tau_0,\psi_{v_0}) \in \mathfrak{ls}({\rm U}_N,{\rm Res}\,{\rm GL}_m,p)$ with $\tau_0$ supercuspidal. Let $S$ be a finite set of places of $F$, containing $v_0$, such that $\pi_v$ is unramified for $v \notin S$. Via Property~(v) of \S~\ref{axioms}, we may assume that $\psi_{v_0}$ is the component of a global additive character $\psi = \otimes \psi_v : F \backslash \mathbb{A}_F \rightarrow \mathbb{C}^\times$.

Via Lemma~3.1 of \cite{LomLS}, there is a cuspidal automorphic representation $\tau = \otimes' \tau_v$ of ${\rm GL}_m(\mathbb{A}_E)$ such that $\tau_{v_0} = \tau_0$ and $\tau_v$ is unramified for all $v \notin S$. Now, using the Grunwald-Wang theorem of class field theory \cite{ArTa2009}, there exists a Gr\"o{\ss}en\-charakter $\eta : E^\times \backslash \mathbb{A}_E^\times \rightarrow \mathbb{C}^\times$ such that $\eta_{v_0} = 1$ and $\eta_v$ is highly ramified for all $v \in S - \left\{ v_0 \right\}$.

At places $v \in S - \left\{ v_0 \right\}$, we have the stable form of Lemma~\ref{stableform}. Indeed, after twisting $\tau_v$ by the highly ramified character $\eta_v$ we have
\begin{equation*}
   \gamma(s,\pi_v \times (\tau_v \cdot \eta_v),\psi_v) = \gamma(s,\Pi_v \times (\tau_v \cdot \eta_v),\psi_v).
\end{equation*}
At split places $v \in S$, we write $\tau_v = \tau_{1,v} \otimes \tau_{2,v}$, regarded as a representation of ${\rm Res}_{E_v/F_v}{\rm GL}_m(F_v) = {\rm GL}_m(F_v) \times {\rm GL}_m(F_v)$. From \S~\ref{separable}~(ii), the Langlands-Shahidi $\gamma$-factors are given by
\begin{equation*}
   \gamma(s,\pi_v \times \tau_v,\psi_v) = \gamma(s,\pi_v \times \tau_{1,v},\psi_v) \gamma(s,\tilde{\pi}_v \times \tau_{2,v},\psi_v),
\end{equation*}
which are compatible with the split base change map of \S~\ref{splitBC}. Now, consider $\tau' = \tau \otimes \eta$. For $(E/F,\pi,\tau',\psi,S) \in \mathcal{LS}({\rm U}_N,{\rm Res}\,{\rm GL}_m,p)$ we have the global functional equation
\begin{equation*}
   L^S(s,\pi \times \tau') = \gamma(s,\pi_{v_0} \times \tau_{v_0},\psi_{v_0}) \prod_{v \in S - \left\{ v_0 \right\}} \gamma(s,\pi_v \times (\tau_v \cdot \eta_v),\psi_v) L^S(1-s,\tilde{\pi} \times \tilde{\tau}'),
\end{equation*}
and for $(E/F,\Pi,\tau',\psi,S) \in \mathcal{LS}({\rm Res}\,{\rm GL}_N,{\rm Res}\,{\rm GL}_m,p)$ the functional equation for Rankin-Selberg products reads
\begin{equation*}
   L^S(s,\Pi \times \tau') = \gamma(s,\Pi_{v_0} \times \tau_{v_0},\psi_{v_0}) \prod_{v \in S - \left\{ v_0 \right\}} \gamma(s,\Pi_v \times (\tau_v \cdot \eta_v),\psi_v) L^S(1-s,\tilde{\Pi} \times \tilde{\tau}').
\end{equation*}
At unramified places, it follows from equation~\eqref{unrBCv} that local $L$-functions agree. This gives equality of the corresponding partial $L$-functions appearing in the above two functional equations. We thus obtain
\begin{equation*}
   \gamma(s,\pi_{v_0} \times \tau_{v_0},\psi_{v_0}) = \gamma(s,\Pi_{v_0} \times \tau_{v_0},\psi_{v_0}).
\end{equation*}
Since our choice of supercuspidal $\tau_0 = \tau_{v_0}$ of ${\rm GL}_m(l)$ was arbitrary, we have that $\Pi_{v_0}$ is a local base change for $\pi_{v_0}$. Uniqueness follows from Lemma~\ref{uniqueBC}. This proves property~(i), as desired.

Note that the central character $\omega_\pi = \otimes \omega_{\pi_v}$ of $\pi$ is an automorphic representation of ${\rm U}_1(\mathbb{A}_F)$. Now, let $\chi = \otimes \chi_v$ be defined on ${\rm GL}_1(\mathbb{A}_E)$ from $\omega_\pi$ via Hilbert's theorem 90. Namely, we let $\mathfrak{h}_v : x_v \mapsto x_v \bar{x}_v^{-1}$ be the continuous map of \eqref{hilbert90} and let
\begin{equation*}
   \chi_v = \omega_{\pi_v} \circ \mathfrak{h}_v
\end{equation*}
at every place $v$ of $F$; we view $E_v$ as a degree-$2$ finite \'etale algebra over $F_v$. Then $\chi : E^\times \backslash \mathbb{A}_E^\times \rightarrow \mathbb{C}^\times$ is a Gr\"o{\ss}en\-charakter. Also, $\chi$ and $\omega_\Pi$ are continuous Gr\"o{\ss}en\-charakters such that $\chi_v$ agrees with $\omega_{\Pi_v}$ at every $v \notin S$. Hence $\chi = \omega_\Pi$. Thus, the central character of $\Pi_{v_0}$ is $\omega_{\Pi_{v_0}} = \chi_{v_0}$, which is obtained via $\omega_{\pi_{v_0}}$ as in property~(ii) of the proposition.
\end{proof}

\begin{remark}\label{globalrmk}
Let $\pi_0$ be a generic representation of ${\rm U}_N(k)$. Suppose there is a globally generic cuspidal automorphic representation $\pi = \otimes' \pi_v$ of ${\rm U}_N(\mathbb{A}_F)$ with $\pi_0 = \pi_{v_0}$ at some place $v_0$ of $F$, where $F_{v_0} = k$. Then it follows from Proposition~\ref{globallocal} that it has a uniquely determined local base change $\Pi_0 = {\rm BC}(\pi_0)$.
\end{remark}

\section{Generic local Langlands functoriality}\label{sec:LLF}

The main purpose of this section is to establish local Langlands functoriality for generic representations of the quasi-split unitary groups; specifically, local stable base change is established in Theorem~\ref{localLC}.

To avoid confusion between local and global notation, throughout this section we use $\pi_0$ to denote a local representation of a unitary group and $\Pi_0$ for its local base change, when it exists; we use $\pi = \otimes' \pi_v$ for a cuspidal automorphic representation of a unitary group and $\Pi = \otimes' \Pi_v$ for its corresponding global base change. Otherwise, we keep the notation of \S~\ref{sec:local}.

\subsection{Supercuspidal lift} We begin with existence and an $L$-functions description for the supercuspidal lift.

\begin{proposition}\label{supercuspidal}
Let $\pi_0$ be a generic supercuspidal representation of ${\rm U}_N(k)$. Then $\pi_0$ has a unique local base change
\begin{equation*}
   \Pi_0 = {\rm BC}(\pi_0)
\end{equation*}
to ${\rm GL}_N(l)$. The central character $\omega_{\Pi_0}$ of $\Pi_0$ is the character of $l^\times$ obtained from $\omega_{\pi_0}$ of $l^1$ via Hilbert's theorem 90. Moreover
\begin{equation*}
   \Pi_0 = {\rm Ind}(\Pi_{0,1} \otimes \cdots \otimes \Pi_{0,d}),
\end{equation*}
where each $\Pi_{0,i}$ is a supercuspidal representation of ${\rm GL}_{N_i}(l)$ satisfying: $\Pi_{0,i} \cong \widetilde{\Pi}_{0,i}^\theta$; $\Pi_{0,i} \ncong \Pi_{0,j}$ for $i \neq j$; and
\begin{itemize}
   \item[(i)] $L(s,\Pi_{0,i},r_\mathcal{A})$ has a pole at $s=0$ if $N$ is odd;
   \item[(ii)] $L(s,\Pi_{0,i} \otimes \eta_{l/k},r_\mathcal{A})$ has a pole at $s=0$ if $N$ is even.
\end{itemize}
\end{proposition}
\begin{proof}
Let $F$ be a global function field with $F_{v_0} = k$. From Lemma~3.1 of \cite{LomLS}, there exists a globally generic cuspidal automorphic representation $\pi = \otimes \pi_v$ of ${\rm U}_N(\mathbb{A}_F)$ such that $\pi_0 = \pi_{v_0}$. Then, Remark~\ref{globalrmk} gives the existence of a unique base change $\Pi_0 = {\rm BC}(\pi_0)$.

Since $\pi_0$ is a generic supercuspidal representation, Proposition~\ref{scTempered} gives
\begin{equation*}
   \Pi_0 = {\rm Ind}(\Pi_{0,1} \otimes \cdots \otimes \Pi_{0,d}),
\end{equation*}
with each $\Pi_{0,i}$ a unitary supercuspidal representation of ${\rm GL}_{N_i}(l)$. In particular, $\Pi_0$ is tempered, hence unitary, and Lemma~\ref{localLfactors} applies to it.

We look at a fixed $\Pi_{0,j}$. We have
\begin{equation*}
   L(s,\pi_0 \times \tilde{\Pi}_{0,j}) = L(s,\Pi_0 \times \tilde{\Pi}_{0,j}) = \prod_{i=1}^d L(s,\Pi_{0,i} \times \tilde{\Pi}_{0,j}).
\end{equation*}
The factor with $i = j$ has a pole at $s = 0$ \cite{JPSS1983}. Since local $L$-functions have no zeros, $L(s,\pi_0 \times \tilde{\Pi}_{0,j})$ has a pole at $s = 0$.

Now, let ${\bf P} = {\bf M}{\bf N}$ be the parabolic subgroup of ${\rm U}_{2N_j+N}$ with Levi ${\bf M} \cong {\rm Res}\,{\rm GL}_{N_j} \times {\rm U}_N$. Then Proposition~5.3 of \cite{LomLS} tells us that $L(s,\pi_0 \times \tilde{\Pi}_{0,j})$ has a pole at $s = 0$ if and only if
\begin{equation*}
   {\rm ind}_P^{{\rm U}_{2N_j+N}(k)}(\tilde{\Pi}_{0,j} \otimes \tilde{\pi}_0)
\end{equation*}
is irreducible and $\tilde{\Pi}_{0,j} \otimes \tilde{\pi}_0 = w_0(\tilde{\Pi}_{0,j} \otimes \tilde{\pi}_0) \cong \Pi_{0,j}^\theta \otimes \tilde{\pi}_0$. This gives, for each $j$, that $\Pi_{0,j} \cong \tilde{\Pi}_{0,j}^\theta$. Furthermore, we have
\begin{equation*}
   L(s,\pi_0 \times \tilde{\Pi}_{0,j}) = L(s,\Pi_0 \times \tilde{\Pi}_{0,j}) = \prod_{i=1}^d L(s,\Pi_{0,i} \times \tilde{\Pi}_{0,j}).
\end{equation*}
Each $L$-function in the product of the right hand side has a pole at $s = 0$ whenever $\Pi_{0,i} \cong \Pi_{0,j}$. However, the pole of $L(s,\pi_0 \times \tilde{\Pi}_{0,j})$ at $s = 0$ being simple, forces $\Pi_{0,i} \ncong \Pi_{0,j}$ for $i \neq j$.

From the fact that $\Pi_{0,i} \cong \tilde{\Pi}_{0,i}^\theta$ and Corollary~\ref{RSAsaicor}, we have
\begin{equation*}
   L(s,\Pi_{0,i} \times \tilde{\Pi}_{0,i}) = L(s,\Pi_{0,i},r_\mathcal{A}) L(s,\Pi_{0,i} \otimes \eta_{l/k},r_\mathcal{A}).
\end{equation*}
Then, one and only one of $L(s,\Pi_{0,i},r_\mathcal{A})$ or $L(s,\Pi_{0,i} \otimes \eta_{l/k},r_\mathcal{A})$ has a simple pole at $s = 0$. For $(l/k,\pi_0,\Pi_{0,i},\psi) \in \mathfrak{ls}({\rm U}_N,{\rm Res}\,{\rm GL}_{N_i},p)$, we have that $\tilde{\Pi}_{0,i} \otimes \tilde{\pi}_0$ is a representation of $M$, with ${\bf M} = {\rm Res}\,{\rm GL}_{N_i} \times {\rm U}_N$ a maximal Levi subgroup of ${\bf G} = {\rm U}_{2N_i+N}$. By Proposition~5.3 of \loccit{LomLS}, the product
\begin{equation*}
   L(s,\tilde{\Pi}_{0,i} \otimes \tilde{\pi}_0,r_1) L(2s,\tilde{\Pi}_{0,i},r_2)
\end{equation*}
has a simple pole at $s = 0$. Using $\Pi_{0,i} \cong \tilde{\Pi}_{0,i}^\theta$ and Corollary~\ref{RSAsaicor}, we have that
\begin{equation*}
   L(s,\Pi_{0,i},r_2) = L(s,\tilde{\Pi}_{0,i},r_2) = \left\{ \begin{array}{ll} L(s,\Pi_{0,i},r_\mathcal{A}) & \text{if } N = 2n \\ L(s,\Pi_{0,i} \otimes \eta_{l/k},r_\mathcal{A}) & \text{if } N = 2n+1 \end{array} \right..
\end{equation*}
Since we showed above that $L(s,\pi_0 \times \tilde{\Pi}_{0,i}) = L(s,\tilde{\Pi}_{0,i} \otimes \tilde{\pi}_0,r_1)$ has a simple pole at $s = 0$, then $L(2s,\tilde{\Pi}_{0,i},r_2)$ cannot. Thus, depending on whether $N$ is even or odd, the other one between $L(s,\Pi_{0,i},r_\mathcal{A})$ and $L(s,\Pi_{0,i} \otimes \eta_{l/k},r_\mathcal{A})$ must have a pole at $s = 0$.
\end{proof}

\subsection{Discrete series, tempered representations} Thanks to the work of M\oe glin and Tadi\'c \cite{MoTa2002}, we have the classification of generic discrete series representations for the unitary groups. Their work allows us to obtain a generic discrete series representation $\xi$ of ${\rm U}_N(k)$ as a subrepresentation as follows
\begin{equation}\label{MTdecomp}
   \xi \hookrightarrow {\rm Ind}(\delta_1 \otimes \cdots \otimes \delta_d \otimes \delta_1' \otimes \cdots \otimes \delta_e' \otimes \pi_0).
\end{equation}
Here, for $1 \leq i \leq d$ and $1 \leq j \leq e$, we have essentially square integrable representations $\delta_i$ of ${\rm GL}_{n_i}(l)$ and $\delta_j'$ of ${\rm GL}_{m_j}(l)$. The representation $\pi_0$ is a generic supercuspidal of ${\rm U}_{N_0}(k)$, with $N_0$ of the same parity as $N$. We refer to \cite{GaLo2018} for a discussion including non-generic representations and the basic assumption (BA) that is made in \cite{MoTa2002}.

We can apply the Zelevinsky classification \cite{Ze1980} to the essentially discrete representations of general linear groups appearing in the decomposition \eqref{MTdecomp}. They are obtained via segments of the form
\begin{equation*}
   \Delta = \left[ \rho\nu^{-b}, \rho\nu^a \right],
\end{equation*}
where $a, b \in \frac{1}{2}\mathbb{Z}$ with $a + b \in \mathbb{Z}_{\geq 0}$, and $\rho$ is a supercuspidal representation of ${\rm GL}_f(l)$, $f | n_i$ or $m_j$, respectively.

The M\oe glin-Tadi\'c classification involves a further refinement of the segments corresponding to each $\delta_i$ and $\delta_j'$. More precisely, let $a_i > b_i > 0$ now be integers of the same parity. Then $\delta_i$ is the generic constituent of
\begin{equation}\label{deltai}
   {\rm Ind}\left( \rho_i\nu^{-\frac{b_i-1}{2}} \otimes \cdots \otimes \rho_i\nu^{\frac{a_i-1}{2}} \right),
\end{equation}
where $\rho_i$ is a supercuspidal representation of ${\rm GL}_f(l)$, $f | n_i$. Furthermore, we have $\rho_i \cong \tilde{\rho}_i^\theta$. Next, let $a_j'$ be a positive integer. We set $\epsilon_j = 1/2$ if $a_j'$ is even and $\epsilon_j = 1$ if $a_j'$ is odd. Then $\delta_j'$ is the generic constituent of
\begin{equation}
   {\rm Ind}\left( \rho_j'\nu^{\epsilon_j} \otimes \cdots \otimes \rho_j'\nu^{\frac{a_j'-1}{2}} \right),
\end{equation}
with $\rho_j' \cong (\tilde{\rho}_j')^\theta$ a supercuspidal representation of ${\rm GL}_{f'}(l)$, $f' | m_j$. Furthermore, the integer $a_j'$ will be odd if $L(s,\pi_0 \times \tilde{\rho}_j')$ has a pole at $s = 0$ (equivalently, $L(s,\tilde{\pi}_0 \times \rho_j') = L(s,\rho_j' \otimes \pi_0,r_1)$, since $\pi_0$ and $\rho_j'$ are unitary), and $a_j'$ will be even otherwise. This is due to (BA) of \cite{MoTa2002} concerning the reducibility of the induced representation ${\rm Ind}(\rho_j'\nu^s \otimes \pi_0)$ at $s = 1/2$ or $1$, which is addressed in Theorem~5.4 of \cite{LomLS} and in \S\S~7.8 and 7.9 of \cite{GaLo2018}.

\begin{proposition}\label{discreteseries}
Let $\xi_0$ be a generic discrete series representation of ${\rm U}_N(k)$, which we can write as
\begin{equation}
   \xi_0 \hookrightarrow {\rm Ind}(\delta_1 \otimes \cdots \otimes \delta_d \otimes \delta_1' \otimes \cdots \otimes \delta_e' \otimes \pi_0)
\end{equation}
with $\pi_0$ a generic supercuspidal of ${\rm U}_{N_0}(k)$ and $\delta_i$, $\delta_j'$ as in \eqref{MTdecomp}. Then $\xi_0$ has a uniquely determined base change
\begin{equation*}
   \Xi_0 = {\rm BC}(\xi_0),
\end{equation*}
which is a tempered generic representation of ${\rm GL}_N(l)$ satisfying $\Xi_0 \cong \tilde{\Xi}_0^\theta$. The central character $\omega_{\Xi_0}$ of $\Xi_0$ is the character of $l^\times$ obtained from $\omega_{\xi_0}$ of $l^1$ via Hilbert's theorem 90. Moreover, the lift $\Xi_0$ is the generic constituent of an induced representation:
\begin{equation*}
   \Xi_0 \hookrightarrow {\rm Ind}\left( \delta_1 \otimes \cdots \otimes \delta_d \otimes \delta_1' \otimes \cdots \otimes \delta_e' \otimes \Pi_0 \otimes \tilde{\delta}_e'^{\,\theta} \otimes \cdots \otimes \tilde{\delta}_1'^{\,\theta} \otimes \tilde{\delta}_d^\theta \otimes \cdots \otimes \tilde{\delta}_1^\theta \right).
\end{equation*}
The representation $\Pi_0$ is the local Langlands functorial lift of $\pi_0$ of Proposition~\ref{supercuspidal}.
\end{proposition}
\begin{proof}
Let $\xi_0$ be as in the statement of the proposition, and consider quadruples $(l/k,\xi_0,\rho,\psi) \in \mathfrak{ls}({\rm U}_N,{\rm Res}\,{\rm GL}_m,p)$ such that $\rho$ is an arbitrary supercuspidal representation. To $\pi_0$ there corresponds a $\Pi_0$ via Proposition~\ref{supercuspidal}. With $\Xi_0$ as in the proposition, we have $(l/k,\Xi_0,\rho,\psi) \in \mathfrak{ls}({\rm Res}\,{\rm GL}_N,{\rm Res}\,{\rm GL}_m,p)$. Then we can use multiplicativity of $\gamma$-factors to obtain
\begin{align*}
   \gamma(s,\xi_0 \times \rho,\psi_l) = & \, \gamma(s,\pi_0 \times \rho,\psi_l) \prod_{i=1}^d \gamma(s,\delta_i \times \rho,\psi_l) \gamma(s,\tilde{\delta}_i^\theta \times \rho,\psi_l) \\
   & \prod_{j=1}^e \gamma(s,\delta_j' \times \rho,\psi_l) \gamma(s,\tilde{\delta}_j'^{\,\theta} \times \rho,\psi_l) \\
   = & \, \gamma(s,\Pi_0 \times \rho,\psi_l) \prod_{i=1}^d \gamma(s,\delta_i \times \rho,\psi_l) \gamma(s,\tilde{\delta}_i^\theta \times \rho,\psi_l) \\
   & \prod_{j=1}^e \gamma(s,\delta_j' \times \rho,\psi_l) \gamma(s,\tilde{\delta}_j'^{\,\theta} \times \rho,\psi_l) \\
   = & \, \gamma(s,\Xi_0 \times \rho,\psi_l).
\end{align*}
From Lemma~\ref{uniqueBC}, we have that $\Xi_0$ is the unique local Langlands lift of $\xi_0$. It satisfies $\Xi_0 \cong \tilde{\Xi}_0^\theta$ and has the right central character.

For each $i$, $1 \leq i \leq d$, let $\tau_i$ be the generic constituent of ${\rm Ind}(\delta_i \otimes \tilde{\delta}_i^\theta)$. After rearranging the factors coming from equation~\eqref{deltai}, we can see that $\tau_i$ is isomorphic to the generic constituent of
\begin{equation*}
   {\rm Ind}\left( (\rho_i\nu^{-\frac{a_i-1}{2}} \otimes \cdots \otimes \rho_i\nu^{\frac{a_i-1}{2}}) \otimes (\rho_i\nu^{-\frac{b_i-1}{2}} \otimes \cdots \otimes \rho_i\nu^{\frac{b_i-1}{2}}) \right).
\end{equation*}
Recall that $\rho_i \cong \tilde{\rho}_i^\theta$. Hence, each $\tau_i$ is tempered and satisfies $\tau_i \cong \tilde{\tau}_i^\theta$. We now turn to the $\delta_j'$, $1 \leq j \leq e$. Since $\pi_0$ is generic, the reducibility point of ${\rm Ind}(\rho_j'\nu^s \otimes \pi_0)$ is $1/2$ or $1$, and hence the supercuspidal representations $\rho_1', \ldots, \rho_e'$ are pairwise non-isomorphic by (BA) of \cite{MoTa2002}.
If $a_j'$ is even, then $\epsilon_j = 1/2$ and the segments of $\delta_j'$ and $\tilde{\delta}_j'^{\,\theta}$, namely
\begin{equation*}
   \left[ \rho_j'\nu^{\frac{1}{2}}, \rho_j'\nu^{\frac{a_j'-1}{2}} \right] \text{ and } \left[ \rho_j'\nu^{-\frac{a_j'-1}{2}}, \rho_j'\nu^{-\frac{1}{2}} \right],
\end{equation*}
are juxtaposed. Their union is the segment $[ \rho_j'\nu^{-\frac{a_j'-1}{2}}, \rho_j'\nu^{\frac{a_j'-1}{2}} ]$, and we let $\tau_j'$ be the unitary discrete series representation attached to it. If $a_j'$ is odd, then $\epsilon_j = 1$, and the union of the two segments is $[ \rho_j'\nu^{-\frac{a_j'-1}{2}}, \rho_j'\nu^{\frac{a_j'-1}{2}} ]$ with the point $\rho_j'$ removed. In this case $L(s,\pi_0 \times \tilde{\rho}_j')$ has a pole at $s = 0$. Writing $\Pi_0 = {\rm Ind}(\Pi_{0,1} \otimes \cdots \otimes \Pi_{0,c})$ as in Proposition~\ref{supercuspidal}, Lemma~\ref{localLfactors} gives
\begin{equation*}
   L(s,\pi_0 \times \tilde{\rho}_j') = L(s,\Pi_0 \times \tilde{\rho}_j') = \prod_{i=1}^c L(s,\Pi_{0,i} \times \tilde{\rho}_j'),
\end{equation*}
which has a pole at $s = 0$ if and only if $\rho_j' \cong \Pi_{0,i}$ for some $i$ \cite{JPSS1983}. Such an $i$ is unique, since the $\Pi_{0,i}$ are pairwise non-isomorphic. Thus the missing point $\rho_j'$ is supplied by $\Pi_0$, and we let $\tau_j'$ be the unitary discrete series representation attached to the full segment $[ \rho_j'\nu^{-\frac{a_j'-1}{2}}, \rho_j'\nu^{\frac{a_j'-1}{2}} ]$. Since the $\rho_j'$ are pairwise non-isomorphic, different indices $j$ use different $\Pi_{0,i}$. Let $\Pi_0'$ be the product of the remaining $\Pi_{0,i}$, a tempered representation.

The supercuspidal support of $\Xi_0$ is then the union of the supercuspidal supports of $\tau_1, \ldots, \tau_d$, $\tau_1', \ldots, \tau_e'$ and $\Pi_0'$. Each of these representations is a product of unitary discrete series representations of general linear groups. Segments of the form $[\rho\nu^{-\frac{t-1}{2}}, \rho\nu^{\frac{t-1}{2}}]$, with $\rho$ unitary, are pairwise not linked, so the product of the corresponding discrete series representations is irreducible and generic \cite{Ze1980}. Since $\Xi_0$ is the unique generic representation with its supercuspidal support, we obtain
\begin{equation*}
   \Xi_0 \cong {\rm Ind}\left( \tau_1 \otimes \cdots \otimes \tau_d \otimes \tau_1' \otimes \cdots \otimes \tau_e' \otimes \Pi_0' \right),
\end{equation*}
which is tempered.
\end{proof}

We now turn to the tempered case, which is crucial, since $L$-functions and $\varepsilon$-factors are defined via $\gamma$-factors in this case. The proofs of the remaining two results are now similar to the case of a discrete series, thanks to Lemma~\ref{uniqueBC}.

\begin{proposition}\label{temperedBC}
Let $\tau_0$ be a generic tempered representation of ${\rm U}_N(k)$, which we can write as
\begin{equation*}
   \tau_0 \hookrightarrow {\rm Ind}\left( \delta_1 \otimes \cdots \otimes \delta_d \otimes \xi_0 \right),
\end{equation*}
with each $\delta_i$ a discrete series representation of ${\rm GL}_{n_i}(l)$ and $\xi_0$ one of ${\rm U}_{N_0}(k)$. Then $\tau_0$ has a uniquely determined base change
\begin{equation*}
   T_0 = {\rm BC}(\tau_0),
\end{equation*}
which is a tempered generic representation of ${\rm GL}_N(l)$ satisfying $T_0 \cong \tilde{T}_0^\theta$. The central character $\omega_{T_0}$ is obtained from $\omega_{\tau_0}$ via Hilbert's theorem 90. Specifically, the lift $T_0$ is of the form
\begin{equation*}
   T_0 = {\rm Ind}\left( \delta_1 \otimes \cdots \otimes \delta_d \otimes \Xi_0 \otimes \tilde{\delta}_d^{\,\theta} \otimes \cdots \otimes \tilde{\delta}_1^\theta \right).
\end{equation*}
The representation $\Xi_0$ is the base change lift of Proposition~\ref{discreteseries}.
\end{proposition}
\begin{proof}
The proof is now along the lines of Proposition~\ref{discreteseries}, where we use multiplicativity of $\gamma$-factors and the fact that $\Xi_0$ is the local Langlands lift of the discrete series $\xi_0$. This way, we obtain equality of $\gamma$-factors to apply Lemma~\ref{uniqueBC} and conclude that $T_0 \cong \tilde{T}_0^\theta$ and has the correct central character.
\end{proof}

\begin{theorem}\label{temperedtempered}
Let $\tau_0$ be a generic tempered representation of ${\rm U}_N(k)$. Then ${\rm BC}(\tau_0)$ is tempered, hence unitary.
\end{theorem}
\begin{proof}
For $\tau_0$ supercuspidal, this is Propositions~\ref{supercuspidal} and \ref{scTempered}. For a generic discrete series representation, it is Proposition~\ref{discreteseries}, whose proof uses only the supercuspidal case and the classification of \cite{MoTa2002}. In general, write $\tau_0$ as in Proposition~\ref{temperedBC}. Then ${\rm BC}(\tau_0) = T_0$ is induced from the unitary discrete series representations $\delta_i$, $\tilde{\delta}_i^\theta$ and the tempered representation $\Xi_0$, hence it is tempered. Tempered representations of ${\rm GL}_N(l)$ are unitary.
\end{proof}

\subsection{Generic local Langlands functoriality} Theorem~\ref{localLC} summarizes our main local result. Local base change is recursively defined via the tempered, discrete series and supercuspidal cases of Propositions~\ref{temperedBC}, \ref{discreteseries} and \ref{supercuspidal}, respectively.

In general, we have the Langlands quotient \cite{BoWa2000}. The standard module conjecture (see \cite{dCHeLo} in characteristic $p$) allows us to realize a general generic representation $\pi_0$ of ${\rm U}_N(k)$ as the unique generic quotient of an induced representation. More precisely, $\pi_0$ is the Langlands quotient of
\begin{equation}\label{langlandsquotient}
   {\rm Ind}\left( \tau_b' \otimes \cdots \otimes \tau_1' \otimes \tau_0 \right),
\end{equation}
with each $\tau_i'$ a quasi-tempered representation of ${\rm GL}_{n_i}(l)$ and $\tau_0$ a tempered representation of ${\rm U}_{N_0}(k)$. We can write $\tau_i' = \tau_{i,0}\nu^{r_i}$ with $\tau_{i,0}$ tempered and the Langlands parameters have $0 < r_1 < \cdots < r_b$. In what follows, we also allow $0 \leq r_1 \leq \cdots \leq r_b$, in which case $\pi_0$ is the generic constituent of \eqref{langlandsquotient}. Merging consecutive factors with equal exponents and absorbing into $\tau_0$ those with $r_i = 0$ recovers the strict form of Axiom~(x); by Lemma~\ref{uniqueBC}, the base change of Theorem~\ref{localLC} does not depend on the choice of such an expression.

\begin{theorem}\label{localLC}
Let $\pi_0$ be a generic representation of ${\rm U}_N(k)$. Write $\pi_0$ as the generic constituent of
\begin{equation*}
   {\rm Ind}\left( \tau_b' \otimes \cdots \otimes \tau_1' \otimes \tau_0 \right),
\end{equation*}
as in \eqref{langlandsquotient}. Then $\pi_0$ has a unique generic local base change
\begin{equation*}
   \Pi_0 = {\rm BC}(\pi_0),
\end{equation*}
which is a generic representation of ${\rm GL}_N(l)$ satisfying $\Pi_0 \cong \tilde{\Pi}_0^\theta$. The central character $\omega_{\Pi_0}$ is obtained from $\omega_{\pi_0}$ via Hilbert's theorem 90. Specifically, the lift $\Pi_0$ is the generic constituent of
\begin{equation*}
   {\rm Ind}\left( \tau_b' \otimes \cdots \otimes \tau_1' \otimes T_0 \otimes \tilde{\tau}_1'^\theta \otimes \cdots \otimes \tilde{\tau}_b'^\theta \right),
\end{equation*}
with $T_0$ the Langlands functorial lift of the tempered representation $\tau_0$. Given $(l/k,\pi_0,\tau,\psi) \in \mathfrak{ls}({\rm U}_N,{\rm Res}\,{\rm GL}_m,p)$, we have equality of local factors
\begin{align*}
   \gamma(s,\pi_0 \times \tau,\psi_l) &= \gamma(s,\Pi_0 \times \tau,\psi_l) \\
   L(s,\pi_0 \times \tau) &= L(s,\Pi_0 \times \tau) \\
   \varepsilon(s,\pi_0 \times \tau,\psi_l) &= \varepsilon(s,\Pi_0 \times \tau,\psi_l).
\end{align*}
\end{theorem}
\begin{proof}
We reason as in the case of a tempered representation, Proposition~\ref{temperedBC}. Equality of $\gamma$-factors follows from the definition of base change and Lemma~\ref{uniqueBC}. For $L$-functions and root numbers, Properties~(ix) and (x) of \S~\ref{axioms}, together with the corresponding properties of Rankin-Selberg factors \cite{HeLo2013a}, reduce the equalities to the factors involving $\tau_0$ and $T_0$, after incorporating twists by unramified characters. Since $T_0$ is tempered by Proposition~\ref{temperedBC}, hence unitary, Lemma~\ref{localLfactors} applies.
\end{proof}

\section{Strong base change}\label{sec:strong}

We now turn towards global Langlands functoriality for generic representations of the quasi-split unitary groups. Keeping the notation of \S~\ref{sec:LLF}, we refine base change so that it agrees with the local functorial lift of Theorem~\ref{localLC} at every place. Let
\begin{equation*}
   \mathfrak{h} : {\rm GL}_1(\mathbb{A}_E) \rightarrow {\rm U}_1(\mathbb{A}_F)
\end{equation*}
be the global reciprocity map such that $\mathfrak{h}_v$ is the map given by Hilbert's theorem 90 at every place $v$ of $F$, as in equation~\eqref{hilbert90}. We also have global twists by the unramified character $\nu$ of a general linear group obtained via the determinant, as in the local theory.

\begin{theorem}\label{strongBC}
Let $\pi$ be a cuspidal globally generic automorphic representation of ${\rm U}_N(\mathbb{A}_F)$. Then $\pi$ has a unique base change to an automorphic representation of ${\rm GL}_N(\mathbb{A}_E)$, denoted by
\begin{equation*}
   \Pi = {\rm BC}(\pi).
\end{equation*}
The central character of $\Pi$ is given by $\omega_\Pi = \omega_\pi \circ \mathfrak{h}$ and is unitary. Furthermore, $\Pi \cong \tilde{\Pi}^\theta$ and there is an expression as an isobaric sum
\begin{equation*}
   \Pi = \Pi_1 \boxplus \cdots \boxplus \Pi_d,
\end{equation*}
where each $\Pi_i$ is a unitary cuspidal automorphic representation of ${\rm GL}_{N_i}(\mathbb{A}_E)$ such that $\tilde{\Pi}_i \cong \Pi_i^\theta$ and $\Pi_i \ncong \Pi_j$, for $i \neq j$. At every place $v$ of $F$, we have that
\begin{equation*}
   \Pi_v = {\rm BC}(\pi_v)
\end{equation*}
is the local base change of Theorem~\ref{localLC} preserving local factors.
\end{theorem}
\begin{proof}
Let $\Pi = {\rm BC}(\pi)$ be the automorphic representation of ${\rm Res}_{E/F}{\rm GL}_N(\mathbb{A}_F) \cong {\rm GL}_N(\mathbb{A}_E)$ obtained from $\pi$ via the weak base change of Proposition~\ref{weakBC}. Let $S$ be a finite set of places of $F$ such that $E/F$ and $\pi$ are unramified outside of $S$. For all $v \notin S$, $\Pi_v$ is the unramified lift of $\pi_v$ of \S~\ref{unramifiedBC} or the split base change of \S~\ref{splitBC}. By Proposition~\ref{globallocal}, we know that $\Pi$ has central character
\begin{equation*}
   \omega_\Pi = \omega_\pi \circ \mathfrak{h},
\end{equation*}
which is unitary, since $\omega_\pi$ is a character of the compact group ${\rm U}_1(F) \backslash {\rm U}_1(\mathbb{A}_F)$.

From the Jacquet-Shalika classification \cite{JaSh1981}, the functorial lift $\Pi$ decomposes as an isobaric sum
\begin{equation*}
   \Pi = \Pi_1 \boxplus \cdots \boxplus \Pi_d.
\end{equation*}
More precisely, each cuspidal automorphic representation $\Pi_i$ of ${\rm GL}_{N_i}(\mathbb{A}_E)$ can be written in the form
\begin{equation*}
   \Pi_i = \Xi_i\nu^{t_i},
\end{equation*}
with $\Xi_i$ a unitary cuspidal automorphic representation of ${\rm GL}_{N_i}(\mathbb{A}_E)$. Reordering if necessary, we can assume $t_1 \leq \cdots \leq t_d$. We wish to prove that each $t_i = 0$. Notice that if $t_i < 0$ for some $i$, then there is a $j > i$ such that $t_j > 0$, due to the fact that $\omega_\Pi$ is unitary. Also, we cannot have $t_1 > 0$. To obtain a contradiction, suppose there exists a $t_{j_0}$ which is the smallest with the property $t_{j_0} < 0$.

Consider $(E/F,\pi,\Xi_{j_0},\psi) \in \mathcal{LS}({\rm U}_N,{\rm Res}\,{\rm GL}_{N_{j_0}},p)$. Theorem~6.5 of \cite{Lom19} tells us that $L(s,\pi \times \tilde{\Xi}_{j_0})$ is holomorphic for $\Re(s) > 1$. However, if we consider $(E/F,\Pi,\Xi_{j_0},\psi) \in \mathcal{LS}({\rm Res}\,{\rm GL}_N,{\rm Res}\,{\rm GL}_{N_{j_0}},p)$, we have that
\begin{equation*}
   L(s,\Pi \times \tilde{\Xi}_{j_0}) = \prod_{i=1}^d L(s+t_i,\Xi_i \times \tilde{\Xi}_{j_0}).
\end{equation*}
And, by Proposition~3.6 of \cite{JaSh1981} part II, the $L$-function
\begin{equation*}
   L(s,\Xi_{j_0} \times \tilde{\Xi}_{j_0})
\end{equation*}
has a simple pole at $s = 1$. This carries over to a pole at $s = 1 - t_{j_0} > 1$ for
\begin{equation*}
   L(s,\pi \times \tilde{\Xi}_{j_0}) = L(s,\Pi \times \tilde{\Xi}_{j_0}).
\end{equation*}
This causes a contradiction unless there exists no such $t_{j_0}$. Hence, all $t_i$ must be zero. This proves that each $\Pi_i$ is a unitary cuspidal automorphic representation.

We now have that each $\Pi_i = \Xi_i$ and
\begin{equation*}
   L(s,\Pi \times \tilde{\Pi}_{i_0}) = \prod_{i=1}^d L(s,\Pi_i \times \tilde{\Pi}_{i_0}),
\end{equation*}
where $i_0$ is a fixed index ranging from $1 \leq i_0 \leq d$. On the right hand side we have a pole at $s = 1$ every time that $\Pi_{i_0} \cong \Pi_i$, by Proposition~3.6 of \cite{JaSh1981} part II. However, on the left hand side, from Theorem~6.5 of \cite{Lom19}, we have that $L(s,\Pi \times \tilde{\Pi}_{i_0}) = L(s,\pi \times \tilde{\Pi}_{i_0})$ has only a simple pole at $s = 1$. Hence, $\Pi_{i_0} \cong \Pi_i$ can only occur if $i_0 = i$.

It remains to show that, for each $j$, we have $\Pi_j \cong \tilde{\Pi}_j^\theta$. For this, let ${\rm T}_j = \Pi_j \otimes \pi$. From Corollary~4.2 of \loccit{Lom19}, $L(s,\pi \times \tilde{\Pi}_j)$ is a Laurent polynomial if $\tilde{w}_0({\rm T}_j) \ncong {\rm T}_j$. If this were the case, it would have no poles. However
\begin{equation*}
   L(s,\pi \times \tilde{\Pi}_j) = \prod_{i=1}^d L(s,\Pi_i \times \tilde{\Pi}_j)
\end{equation*}
has only a simple pole at $s = 1$, reasoning as before. Thus, it must be the case that $\tilde{w}_0({\rm T}_j) \cong {\rm T}_j$, giving the desired $\Pi_j \cong \tilde{\Pi}_j^\theta$.

The compatibility of global to local base change is addressed by Proposition~\ref{globallocal}~(i). Indeed, we have that
\begin{equation*}
   \Pi_v = {\rm BC}(\pi_v)
\end{equation*}
is unique and must be given by Theorem~\ref{localLC} at every place. Moreover, since each $\Pi_i$ is a unitary cuspidal automorphic representation, the induced representation ${\rm Ind}(\Pi_{1,v} \otimes \cdots \otimes \Pi_{d,v})$ is irreducible by a corollary of Bernstein \cite{Be1984}; hence, it is isomorphic to $\Pi_v$, which is thus unitary. Since $\Pi_j \cong \tilde{\Pi}_j^\theta$ for each $j$, we also have $\Pi \cong \tilde{\Pi}^\theta$ globally.
\end{proof}

In the case of number fields, the method of descent is used in \cite{So2005} to show how to obtain the strong lift from the weak lift. Over function fields, we have given a self contained proof with the results of this article combined with those of \cite{Lom19}. Furthermore, we note that Theorem~6.5 of \loccit{Lom19}, which we use in the proof of Theorem~\ref{strongBC}, is somewhat simpler to establish for the quasi-split unitary groups than for the split classical groups, thanks to the presence of split places. At a place $v$ of $F$ that splits in $E$, the local component $\pi_v$ is a representation of ${\rm GL}_N(F_v)$, and unitary representations of general linear groups have a simpler description.

\section{Ramanujan Conjecture and Riemann Hypothesis}\label{sec:RR}

We complete the proof of existence for extended $\gamma$-factors, $L$-functions and root numbers in order to include products of two unitary groups for generic representations. In \S~\ref{nongeneric}, assuming the tempered $L$-packet conjecture, we obtain local factors on all of $\mathfrak{ls}(p)$. Our main results on global $L$-functions are contained in Theorem~\ref{mainthm}, which is proved in \S~\ref{proofmain}. We conclude with Theorem~\ref{ramanujan}, which establishes the Ramanujan Conjecture.

\begin{theorem}\label{mainthm}
There exist rules $\gamma$, $L$ and $\varepsilon$ for generic $(l/k,\pi,\tau,\psi) \in \mathfrak{ls}(p)$, satisfying the ten axioms of \S~\ref{axioms}. Given $(E/F,\pi,\tau,\psi) \in \mathcal{LS}(p)$, define
\begin{equation*}
   L(s,\pi \times \tau) = \prod_v L(s,\pi_v \times \tau_v) \quad \text{and} \quad \varepsilon(s,\pi \times \tau,\psi) = \prod_v \varepsilon(s,\pi_v \times \tau_v,\psi_v).
\end{equation*}
Automorphic $L$-functions on $\mathcal{LS}(p)$ satisfy the following:
\begin{itemize}
   \item[(i)] {\rm (Rationality).} $L(s,\pi \times \tau)$ converges absolutely on a right half plane and has a meromorphic continuation to a rational function in $q^{-s}$.
   \item[(ii)] {\rm (Functional equation).} $L(s,\pi \times \tau) = \varepsilon(s,\pi \times \tau) L(1-s,\tilde{\pi} \times \tilde{\tau})$.
   \item[(iii)] {\rm (Riemann Hypothesis).} The zeros of $L(s,\pi \times \tau)$ are contained in the line $\Re(s) = 1/2$.
\end{itemize}
\end{theorem}

\subsection{Extended local factors} Let us complete the definition of extended local factors for generic representations. The case of a unitary group and a general linear group for generic representations is included in Theorem~5.1 of \cite{LomLS}. An exposition, within the Langlands-Shahidi method, of the case of two general linear groups can be found in \cite{HeLo2013a}. We thus assume ${\bf G}_1 = {\rm U}_M$ and ${\bf G}_2 = {\rm U}_N$.

\begin{definition}
Given $(l/k,\pi_0,\tau_0,\psi) \in \mathfrak{ls}({\rm U}_M,{\rm U}_N,p)$ generic, let
\begin{equation*}
   \Pi_0 = {\rm BC}(\pi_0) \text{ and } T_0 = {\rm BC}(\tau_0)
\end{equation*}
be the corresponding base change lifts of Theorem~\ref{localLC}. We define
\begin{align*}
   \gamma(s,\pi_0 \times \tau_0,\psi_l) &:= \gamma(s,\Pi_0 \times T_0,\psi_l) \\
   L(s,\pi_0 \times \tau_0) &:= L(s,\Pi_0 \times T_0) \\
   \varepsilon(s,\pi_0 \times \tau_0,\psi_l) &:= \varepsilon(s,\Pi_0 \times T_0,\psi_l).
\end{align*}
\end{definition}

The defining Properties (vii)--(x) of \S~\ref{axioms} allow us to construct $L$-functions and root numbers from $\gamma$-factors in the tempered case. This is compatible with the decomposition of $\pi_0$ and $\tau_0$ of Theorem~\ref{localLC}. It is then not hard to verify that the rules $\gamma$, $L$ and $\varepsilon$ satisfy all of the properties of \S~\ref{axioms}. For instance, multiplicativity, Property~(iv), and Properties~(vii)--(x) follow from Theorem~\ref{localLC} and the corresponding properties of Rankin-Selberg factors \cite{JPSS1983, HeLo2013a}, while the global functional equation, Property~(vi), follows from Theorem~\ref{strongBC} and the functional equation for Rankin-Selberg products.

\subsection{Non-generic representations and local factors}\label{nongeneric} Consider an irreducible admissible representation $\pi$ of ${\rm U}_N(k)$ that is not necessarily generic. We first look at the case when $\pi$ is tempered. If ${\rm char}(k) = 0$, the tempered $L$-packet conjecture, Conjecture~\ref{temperedconj} below, is known (see Theorem~2.5.1 of \cite{Mo2015}). Furthermore, it is part of the work of Ganapathy-Varma \cite{GaVa2017} on the local Langlands correspondence for the split classical groups if ${\rm char}(k) = p$. For the unitary groups ${\rm U}_N(k)$, ${\rm char}(k) = p$, it is thus reasonable to work under the assumption that Conjecture~\ref{temperedconj} holds.

Let $\Phi$ be the set of all Langlands parameters
\begin{equation*}
   \phi : \mathcal{W}_k' \rightarrow {}^L{\rm U}_N.
\end{equation*}
A parameter $\phi$ is tempered if its image on ${\rm GL}_N(\mathbb{C})$ is bounded. For any tempered $L$-parameter $\phi$, there is an $L$-packet $\Pi_\phi$ which is a finite multi-set. We consider only tempered $L$-packets in this section, which agree with tempered Arthur packets.

\begin{conjecture}\label{temperedconj}
If an $L$-packet contains a tempered element, then all of its elements are tempered. Every tempered $L$-packet $\Pi_\phi$ of ${\rm U}_N(k)$ contains a representation $\pi_0$ which is generic.
\end{conjecture}

In fact, one expects more precisely the following. Write ${\bf G} = {\rm U}_N$ and fix a Borel subgroup ${\bf B} = {\bf T}{\bf U}$. Consider a Whittaker datum $(B,\psi_U)$ with $\psi_U : U \rightarrow \mathbb{C}^\times$ non-degenerate. Then every tempered $L$-packet $\Pi_\phi$ of $G$ should contain exactly one representation $\pi_0$ which is generic with respect to $(B,\psi_U)$. In this article, we take the Borel subgroup of ${\rm U}_N$ consisting of upper triangular matrices.

Furthermore, $L$-functions are independent of how the non-degenerate character varies. Let $\widetilde{\bf G}$ be as in Proposition~2.2 of \cite{LomLS}, sharing the same derived group as ${\bf G}$. Then $L$-functions are invariant under conjugation by elements of $\widetilde{\bf G}(k)$. Furthermore, there is a formula that keeps track of how $\gamma$-factors and root numbers behave as the additive character $\psi$ varies.

The importance of the tempered $L$-packet conjecture is that it reduces the study of local $\gamma$-factors, $L$-functions and $\varepsilon$-factors to the case of generic representations. Hence, under the assumption that Conjecture~\ref{temperedconj} is valid, we first extend the existence part of Theorem~\ref{mainthm} to all tempered $(l/k,\pi,\tau,\psi) \in \mathfrak{ls}(p)$.

\begin{definition}
Let $(l/k,\pi,\tau,\psi) \in \mathfrak{ls}(p)$ be tempered. Let $\Pi_{\phi_1}$ and $\Pi_{\phi_2}$ be tempered $L$-packets with $\pi \in \Pi_{\phi_1}$ and $\tau \in \Pi_{\phi_2}$. Let $\pi_0 \in \Pi_{\phi_1}$ and $\tau_0 \in \Pi_{\phi_2}$ be generic. We define
\begin{align*}
   \gamma(s,\pi \times \tau,\psi_l) &:= \gamma(s,\pi_0 \times \tau_0,\psi_l) \\
   L(s,\pi \times \tau) &:= L(s,\pi_0 \times \tau_0) \\
   \varepsilon(s,\pi \times \tau,\psi_l) &:= \varepsilon(s,\pi_0 \times \tau_0,\psi_l).
\end{align*}
\end{definition}

To have local factors for all admissible $(l/k,\pi,\tau,\psi) \in \mathfrak{ls}(p)$, we can use the Langlands classification to write
\begin{equation*}
   \pi \hookrightarrow {\rm Ind}(\sigma_1 \otimes \chi_1) \quad \text{and} \quad \tau \hookrightarrow {\rm Ind}(\sigma_2 \otimes \chi_2),
\end{equation*}
with $\sigma_i$ a tempered representation of $M_i = {\bf M}_i(k)$ and $\chi_i \in X_{\rm nr}({\bf M}_i)$, the group of unramified characters of $M_i$, in the negative Weyl chamber, as in \cite{BoWa2000}. With tempered $L$-functions and corresponding local factors defined, Properties (vii)--(x) of \S~\ref{axioms} can now be used to define $L$-functions and related local factors on $\mathfrak{ls}({\bf G}_1,{\bf G}_2,p)$ in general.

We refer to \S\S~7 and 8 of \cite{GaLo2018} for a further discussion on $L$-parameters and the local Langlands correspondence for the classical groups, including the unitary groups. Written under certain working hypotheses, we address the local Langlands correspondence in characteristic $p$. First for supercuspidal representations, then for discrete series and tempered $L$-parameters, to end with general admissible representations.

For any $\pi$ of ${\rm U}_N(k)$, Genestier and V.~Lafforgue \cite{GeLa} attach a semisimple $L$-parameter $\sigma_\pi^{ss}$. For generic $\pi$, the Langlands parameter of ${\rm BC}(\pi)$ provided by Theorem~\ref{localLC} and the local Langlands correspondence for ${\rm GL}_N$ is expected to have semisimplification ${\rm BC} \circ \sigma_\pi^{ss}$.

\subsection{Proof of Theorem~\ref{mainthm}}\label{proofmain} The case of two general linear groups, i.e., for $(E/F,\pi,\tau,\psi) \in \mathcal{LS}({\rm Res}_{E/F}{\rm GL}_M,{\rm Res}_{E/F}{\rm GL}_N,p)$, is already well understood. Properties (i) and (ii) of Theorem~\ref{mainthm} are attributed to Piatetski-Shapiro in this case \cite{PS1976}. They can be proved in a self contained way via the Langlands-Shahidi method over function fields, see \cite{HeLo2013a, Lo2016}. And the Riemann Hypothesis for general linear groups was proved by Laurent Lafforgue in \cite{La2002}.

For the case of $(E/F,\pi,\tau,\psi) \in \mathcal{LS}({\rm U}_N,{\rm Res}_{E/F}{\rm GL}_m,p)$: rationality can be found in \cite{Lom19}; Theorem~\ref{mainthm}(ii) is included in Theorem~5.1 of \cite{LomLS}; and, the Riemann Hypothesis follows from Theorem~5.7 of \loccit{LomLS}, since we establish the Ramanujan Conjecture for ${\rm U}_N$ in the next subsection.

For the case of $(E/F,\pi,\tau,\psi) \in \mathcal{LS}({\rm U}_M,{\rm U}_N,p)$, let
\begin{equation*}
   {\rm BC}(\pi) = \Pi = \Pi_1 \boxplus \cdots \boxplus \Pi_d \quad \text{and} \quad {\rm BC}(\tau) = T = T_1 \boxplus \cdots \boxplus T_e
\end{equation*}
be the base change lifts of Theorem~\ref{strongBC}. Then
\begin{equation*}
   L(s,\pi \times \tau) = L(s,\Pi \times T) = \prod_{i,j} L(s,\Pi_i \times T_j).
\end{equation*}
For each $(E/F,\Pi_i,T_j,\psi) \in \mathcal{LS}({\rm Res}_{E/F}{\rm GL}_{m_i},{\rm Res}_{E/F}{\rm GL}_{n_j},p)$, $1 \leq i \leq d$, $1 \leq j \leq e$, we have rationality, the functional equation
\begin{equation*}
   L(s,\Pi_i \times T_j) = \varepsilon(s,\Pi_i \times T_j) L(1-s,\tilde{\Pi}_i \times \tilde{T}_j),
\end{equation*}
and the Riemann Hypothesis. Hence the $L$-function $L(s,\pi \times \tau)$ also satisfies Properties (i)--(iii) of Theorem~\ref{mainthm}. \qed

\subsection{The Ramanujan Conjecture} Base change over function fields also allows us to transport the Ramanujan Conjecture from the unitary groups to ${\rm GL}_N$. The Ramanujan Conjecture for cuspidal representations of general linear groups is a theorem of L.~Lafforgue \cite{La2002}. For globally generic representations of the split classical groups, see \cite{Lo2009, Lom19}; for the split exceptional groups, see the recent preprint \cite{OpenAI2026}.

\begin{theorem}\label{ramanujan}
Let $\pi = \otimes' \pi_v$ be a globally generic cuspidal automorphic representation of ${\rm U}_N(\mathbb{A}_F)$. Then every $\pi_v$ is tempered. Whenever $\pi_v$ is unramified, its Satake parameters have absolute value $1$.
\end{theorem}
\begin{proof}
Fix a place $v$ of $F$, which remains inert in $E$. We can write $\pi_v$ as the generic constituent of
\begin{equation*}
   {\rm Ind}\left( \tau_{b,v}' \otimes \cdots \otimes \tau_{1,v}' \otimes \tau_{0,v} \right),
\end{equation*}
as in \eqref{langlandsquotient}, with each $\tau_{i,v}'$ quasi-tempered and $\tau_{0,v}$ tempered. Furthermore, we can write $\tau_{i,v}' = \tau_{i,v}\nu^{r_{i,v}}$ with $\tau_{i,v}$ tempered and Langlands parameters
\[ 0 \leq r_{1,v} \leq \cdots \leq r_{b,v}. \]

Now, let $\Pi = {\rm BC}(\pi)$ be the base change lift of Theorem~\ref{strongBC}, so that $\Pi_v$ is the local Langlands functorial lift of $\pi_v$. By Theorem~\ref{localLC}, $\Pi_v$ is the generic constituent of
\begin{equation}\label{Piv}
   {\rm Ind}\left( \tau_{b,v}' \otimes \cdots \otimes \tau_{1,v}' \otimes T_{0,v} \otimes \tilde{\tau}_{1,v}'^\theta \otimes \cdots \otimes \tilde{\tau}_{b,v}'^\theta \right),
\end{equation}
with $T_{0,v}$ the Langlands functorial lift of the tempered representation $\tau_{0,v}$. The representation $T_{0,v}$ is also tempered by Proposition~\ref{temperedBC}.

On the other hand, the base change lift can be expressed as an isobaric sum
\begin{equation*}
   \Pi = \Pi_1 \boxplus \cdots \boxplus \Pi_e,
\end{equation*}
with each $\Pi_i$ a cuspidal unitary automorphic representation of ${\rm GL}_{m_i}(\mathbb{A}_E)$. Hence, by Bernstein \cite{Be1984}, $\Pi_v$ is isomorphic to the irreducible representation
\begin{equation}\label{Pivisobaric}
   {\rm Ind}(\Pi_{1,v} \otimes \cdots \otimes \Pi_{e,v}).
\end{equation}
Then, thanks to Th\'eor\`eme~VI.10 of \cite{La2002}, each $\Pi_{i,v}$ is tempered.

We then look at $\tau_{b,v}^\theta$, where $r_{b,v}$ is the largest of the exponents in \eqref{Piv}, and suppose that $r_{b,v} > 0$. We now have
\begin{align*}
   L(s,\pi_v \times \tau_{b,v}^\theta) &= L(s,\Pi_v \times \tau_{b,v}^\theta) \\
   &= L(s,T_{0,v} \times \tau_{b,v}^\theta) \prod_{i=1}^b L(s+r_{i,v},\tau_{i,v} \times \tau_{b,v}^\theta) L(s-r_{i,v},\tilde{\tau}_{i,v}^\theta \times \tau_{b,v}^\theta).
\end{align*}
The $L$-function $L(s-r_{b,v},\tilde{\tau}_{b,v}^\theta \times \tau_{b,v}^\theta)$ appearing on the right hand side has a pole at $s = r_{b,v}$. However, from \eqref{Pivisobaric} we have
\begin{equation*}
   L(s,\Pi_v \times \tau_{b,v}^\theta) = \prod_{i=1}^e L(s,\Pi_{i,v} \times \tau_{b,v}^\theta).
\end{equation*}
Notice that each representation involved in the product on the right hand side is tempered. Then each $L(s,\Pi_{i,v} \times \tau_{b,v}^\theta)$ is holomorphic for $\Re(s) > 0$. Hence, so is $L(s,\Pi_v \times \tau_{b,v}^\theta)$. This contradiction shows that $r_{b,v} = 0$. Thus $\pi_v$ is the generic constituent of a representation induced from tempered data, and it is therefore tempered.

If $v$ is split, we have that ${\rm BC}(\pi_v) = \pi_v \otimes \tilde{\pi}_v$ from \S~\ref{splitBC}. For the place $w_1$ of $E$ lying above $v$ corresponding to the first factor, the representation of ${\rm GL}_N(F_v)$
\begin{equation*}
   \pi_v = {\rm Ind}(\Pi_{1,w_1} \otimes \cdots \otimes \Pi_{e,w_1})
\end{equation*}
is tempered. In particular, if $\pi_v$ is unramified, its Satake parameters have absolute value $1$.

Now, if $v$ is inert and $\pi_v$ is unramified, we have from \S~\ref{unramifiedBC} that the unramified base change ${\rm BC}(\pi_v) = \Pi_v$ corresponds to a semisimple conjugacy class given by
\begin{equation*}
   \Phi_v({\rm Frob}_v) = \left\{ \begin{array}{ll} {\rm diag}(\alpha_{1,v}, \ldots, \alpha_{n,v}, 1, \alpha_{n,v}^{-1}, \ldots, \alpha_{1,v}^{-1}) & \text{if } N = 2n+1 \\ {\rm diag}(\alpha_{1,v}, \ldots, \alpha_{n,v}, \alpha_{n,v}^{-1}, \ldots, \alpha_{1,v}^{-1}) & \text{if } N = 2n \end{array} \right..
\end{equation*}
Each $\alpha_{j,v}$ or $\alpha_{j,v}^{-1}$ is a Satake parameter for one of the representations $\Pi_{i,v}$, which are unramified. Since we are in the case of ${\rm GL}_{m_i}(E_v)$, we conclude that
\begin{equation*}
   \left| \alpha_{j,v} \right| = 1.
\end{equation*}
\end{proof}

\end{document}